\documentclass[12pt]{amsart}
\usepackage[utf8]{inputenc}
\usepackage{colordvi}
\usepackage[dvipsnames]{xcolor}
\usepackage[hypertexnames=false,
 	breaklinks=true,
 	colorlinks=true,
 	linkcolor=blue,
 	citecolor=red,
 	urlcolor=blue,
 ]{hyperref}
\usepackage{amsmath,amssymb,amsthm,mathtools}
\usepackage{enumitem}
\usepackage{fullpage}
\usepackage{cleveref}
\usepackage{wrapfig}
\usepackage{tikz}
\usepackage{caption}
\usetikzlibrary{patterns}
\usepackage[numbers,sort]{natbib}
\usepackage[ruled,linesnumbered, inoutnumbered, algosection]{algorithm2e}

\usepackage{subfiles} 

\usepackage{booktabs}

\newtheorem{theorem}{Theorem}

\newtheorem{proposition}[theorem]{Proposition}

\theoremstyle{definition}
\newtheorem{definition}[theorem]{Definition}
\newtheorem{example}[theorem]{Example}
\newtheorem{remark}[theorem]{Remark}

\newcommand{\CC}{\mathbb{C}}
\newcommand{\IC}{\mathbb{IC}}
\newcommand{\IR}{\mathbb{IR}}
\newcommand{\KK}{\mathbb{K}}
\newcommand{\LL}{\mathbb{L}}
\newcommand{\NN}{\mathbb{N}}
\newcommand{\PP}{\mathbb{P}}

\newcommand{\RR}{\mathbb{R}}
\newcommand{\ZZ}{\mathbb{Z}}

\newcommand{\calA}{\mathcal{A}}
\newcommand{\calC}{\mathcal{C}}

\newcommand{\bfalpha}{{\boldsymbol{\alpha}}}
\newcommand{\bfgamma}{{\boldsymbol{\gamma}}}
\newcommand{\bfa}{\mathbf{a}}
\newcommand{\bfc}{\mathbf{c}}
\newcommand{\bfI}{\mathbf{I}}
\newcommand{\bfp}{\mathbf{p}}

\newcommand{\bfu}{\mathbf{u}}
\newcommand{\bfv}{\mathbf{v}}
\newcommand{\bfx}{\mathbf{x}}
\newcommand{\bfy}{\mathbf{y}}
\newcommand{\bfz}{\mathbf{z}}

\newcommand{\vol}{{\mathrm{vol}}}
\newcommand{\conv}{{\mathrm{conv}}}
\newcommand{\codim}{{\mathrm{codim}}}
\newcommand{\MV}{{\mathrm{MV}}}

\def\calP{\mathcal{P}}
\def\calM{\mathcal{M}}
\def\SE{\mathrm{SE}}
\def\SO{\mathrm{SO}}

\definecolor{myblue}{RGB}{0,0,204}
\definecolor{myred}{RGB}{204,0,0}
\definecolor{mymaroon}{RGB}{125,0,0}

\newcommand{\defcolor}[1]{{\color{blue}#1}}
\newcommand{\demph}[1]{\defcolor{{\sl #1}}}

\title{Numerical Nonlinear Algebra}
\author[Bates]{Daniel J.~Bates}
\address{Daniel J.~Bates, Department of Mathematics, U.S. Naval Academy, 572C Holloway Road, Mail Stop 9E, Annapolis, MD 21402\\ USA}
\email{dbates@usna.edu}
\author[Breiding]{Paul Breiding}
\address{Paul Breiding, Universit\"at Osnabr\"uck, FB Mathematik/Informatik,
Albrechtstr.\ 28a, 49076 Osnabrück,  Germany}
\email{pbreiding@uni-osnabrueck.de}
\urladdr{http://www.paulbreiding.org}
\author[Chen]{Tianran Chen}
\address{Tianran Chen\\
    Department of Mathematics\\
    Auburn University at Montgomery\\
    Montgomery\\
    Alabama 36116\\ USA}
\email{ti@nranchen.org}
\urladdr{http://www.tianranchen.org}

\author[Hauenstein]{Jonathan~D. Hauenstein}
\address{Jonathan D. Hauenstein\\
Department of Applied \& Computational Mathematics \& Statistics\\
         University of Notre Dame\\
         Notre Dame, IN  46556\\
         USA}
\email{hauenstein@nd.edu}
\urladdr{{http://www.nd.edu/~jhauenst}}
\author[Leykin]{Anton Leykin}
\address{Anton Leykin\\
         School of Mathematics\\
         Georgia Institute of Technology\\
         686 Cherry Street\\
         Atlanta, GA 30332-0160\\
         USA}
\email{leykin@math.gatech.edu}
\author[Sottile]{Frank Sottile}
\address{Frank Sottile\\
         Department of Mathematics\\
         Texas A\&M University\\
         College Station\\
         Texas \ 77843\\
         USA}
\email{sottile@tamu.edu}
\urladdr{https://franksottile.github.io/}

\begin{document}

\begin{abstract}
Numerical nonlinear algebra is a computational paradigm that uses numerical analysis
to study polynomial equations.  Its origins were methods to solve systems of
polynomial equations based on the classical theorem of B\'ezout.  This was decisively
linked to modern developments in algebraic geometry  by the polyhedral homotopy
algorithm of Huber and Sturmfels, which exploits the combinatorial structure of the
equations and led to efficient software for solving polynomial equations.

Subsequent growth of numerical nonlinear algebra continues to be informed by algebraic geometry and its
applications.  These include new approaches to solving, algorithms for studying
positive-dimensional varieties, certification, and a range of applications
both within mathematics and from other disciplines.
With new implementations, numerical nonlinear algebra
is now a fundamental computational tool for algebraic geometry and its
applications.
We survey  some of these innovations and some recent applications.
\end{abstract}

\maketitle

\dedicatory{In honor of Bernd Sturmfels on his 60th birthday}

%
\bigskip
\section{Introduction}\label{sec:intro}
Bernd Sturmfels has a knack for neologisms, minting memorable mathematical terms that pithily portray their essence
and pass into general use.
Nonlinear algebra~\cite{NLA_book} is a Sturmfelsian neologism expressing the focus on computation in applications of algebraic geometry, the
objects that appear in applications, and the theoretical underpinnings this inquiry requires.
Numerical nonlinear algebra is numerical computation supporting nonlinear algebra.
It is complementary to symbolic computation (also a key input to nonlinear algebra), and its development has opened up new
vistas to explore and challenges to overcome.

Sturmfels did not create this field, but his work with Huber introducing the polyhedral homotopy algorithm~\cite{HuSt95} catalyzed it.
This algorithm exemplifies Sturmfels' mathematical contributions, exploiting links between algebra and geometric combinatorics
to address problems in other areas of mathematics, in this case the ancient problem of solving equations.
He was also important for its development with his encouragement of researchers, early decisive use of its
methods~\cite{Sturmfels2002,BHORS}, and by popularizing it~\cite{3264}.

We survey some of the main tools in this field, focusing on homotopy continuation.
Our goal is to provide a reasonably comprehensive introduction to some of the main ideas in numerical nonlinear algebra. 
In Section~\ref{S:phc} we describe polynomial homotopy continuation and its basic use to solve systems of polynomial equations.
We develop the background and present some details of the polyhedral homotopy algorithm in Section~\ref{polyhedral-homotopy}.
Numerical algebraic geometry, which uses these tools to represent
and manipulate 
algebraic varieties on a computer, is presented in Section~\ref{sec:nag},
along with new methods for solving equations that this perspective affords.
A welcome and perhaps surprising feature is that there  are often methods to certify the approximate solutions these algorithms provide,
which is sketched in Section~\ref{sec:cert}.
We close this survey by presenting in Section~\ref{sec:apps} three domains in which numerical nonlinear algebra has recently been applied.

%
\section{What is polynomial homotopy continuation?}\label{S:phc}
\demph{Polynomial homotopy continuation} is a numerical method for computing complex solutions to systems of polynomial equations,
 \begin{equation}\label{system}
   F(\textbf{x})\ =\ F(x_1,\ldots,x_n)\ =\ \begin{bmatrix}
   \ f_1(x_1,\ldots,x_n)\ \\
   \vdots \\
   \ f_m(x_1,\ldots,x_n)\
   \end{bmatrix}\ =\ \textbf{0}\,,
 \end{equation}
where $f_i(x_1,\ldots,x_n)\in \CC[x_1,\ldots,x_n]$ for $1\leq i\leq m$.
A point $\bfz\in\CC^n$ is a \demph{zero}, or a \demph{solution}, of $F$ if $F(\bfz) = 0$.
A solution is \demph{regular}  if the Jacobian matrix of the partial derivatives of~$F$ evaluated at $\bfz$ has rank $n$.
Necessarily, $m\geq n$.
When $m=n$, we refer to the system as \demph{square}.
Otherwise, the system is \demph{overdetermined}.

The underlying idea of polynomial homotopy continuation is simple:
to solve $F(\bfx)=0$, we construct another system $G(\bfx)=0$ of polynomial equations with known zeros,
together with a \demph{homotopy}.
This is a family  of polynomial systems $H(\bfx,t)$ that depend algebraically on $t\in\CC$
and which interpolates between $F$ and $G$ in that $H(\bfx,0)=F(\bfx)$ and $H(\bfx,1)=G(\bfx)$.
Considering one zero, $\bfy$, of $G$ and restricting to $t\in[0,1]$, the equation $H(\bfx,t)=0$ defines a \demph{solution path}
$\bfx(t)\subset\CC^n$ such that $H(\bfx(t),t)=0$ for $t\in[0,1]$ and $\bfx(1)=\bfy$.
The path is followed (see Section~\ref{sec:tracking}) from $t=1$ to $t=0$ to compute an approximation of the solution $\bfz=\bfx(0)$ to $F$.
Following this solution path is equivalent to solving the initial value problem
%
 \begin{equation}
    \label{Davidenko}
    \frac{\partial}{\partial \bfx} H(\bfx,t) \, \Big(\, \frac{\mathrm d}{\mathrm d t}\bfx(t)\, \Big)
     + \frac{\partial}{\partial t} H(\bfx,t)\ =\ 0\,,\qquad \bfx(1)\ =\ \bfy\,.
 \end{equation}

This \demph{Davidenko differential equation}~\cite{Davidenko,Davidenko_full} is typically solved using a 
predictor-corrector scheme (see Section~\ref{sec:tracking}).
We say that $\bfx(1)=\bfy$ gets \demph{tracked} towards $\bfx(0)$.
For this predictor-corrector scheme to successfully compute $\bfx(0)$, $\bfx(t)$ must be a regular zero of $H(\bfx,t)=0$ for every
$t\in(0,1]$. 
Nonregular solutions at $t=0$ are treated using specialized numerical methods called
\demph{endgames}~\cite{Morgan:Sommese:Wampler:1992a}.

So far, there is nothing specific about polynomials.
When $F$ is a system of polynomials, we can construct a \demph{start system} $G$ with known zeros and a homotopy $H$
such that for every isolated zero $\bfz$ of~$F$, there is at least one zero of $G$ that can be tracked towards~$\bfz$.
That is, we can compute approximations of all isolated zeros of $F$.

The ideas behind this method were developed over many years, and by a number of people.
Garc\'ia and Zangwill~\cite{GZ79} proposed polynomial homotopy continuation and a classic reference is Morgan's book~\cite{Morgan}.
The textbook by Sommese and Wampler~\cite{Sommese:Wampler:2005} is now a standard reference.
Decades of conceptual advances have spawned several implementations of the core algorithms.
Historically, the first implementation with wide use was \texttt{PHCpack} \cite{PHCpack}, followed a decade later by
\texttt{Bertini} \cite{Bertini}, which is also widely used.
Later came the \texttt{HOM4PS} family \cite{HOM4PSArticle,LeeLiTsai2008HOM4PS},
\texttt{NAG4M2}~\cite{NumericalAlgebraicGeometryArticle}, and \texttt{HomotopyContinuation.jl}~\cite{HC.jl}.
The package \texttt{NumericalAlgebraicGeometry} of \texttt{Macaulay2}~\cite{M2} (referred to as
\texttt{NAG4M2}) includes interfaces to many of the alternatives listed above, for example, see~\cite{NAG4M2_PHCpack}.

We next explain polynomial homotopy continuation with more care and in more detail.
We first discuss what is meant by numerical method and approximate zero of a system.

\subsection{The solution to a system of polynomial equations}

A solution to the system~\eqref{system} is a point $\bfz\in\CC^n$ satisfying~\eqref{system}.
The collection of all such points is an \demph{algebraic variety},
\begin{equation}\label{eq:AlgVariety}
    V\ \vcentcolon=\ \{\bfz\in\CC^n \mid f_1(\bfz)=\cdots = f_m(\bfz)=0\}\,.
\end{equation}
This defines solutions $\bfz$ \demph{implicitly}, using just the definition of $F$.
It is hard to extract any information other than ``$\bfz$ is a solution to $F$'' from this representation.
A more useful representation of $V$ is given by a \demph{Gr\"obner basis} \cite{CLO, Sturmfels2002, Sturmfels2005}.

Consider a simple example of two polynomial equations in two variables,
 \begin{equation}\label{Eq:Ex2.1}
    x^2+y^2-1\ =\ 0, \ x^2-y^3-y-1\ =\ 0\,,
 \end{equation}
describing the intersection of two plane curves.
A Gr\"obner basis is $\{y^3+y^2+y,\ x^2+y^2-1\}$.
Its triangularity facilitates solving.
The first equation, $y^3+y^2+y=0$, has the three solutions $0, (-1\pm \sqrt{-3})/2$.
Substituting each into the second gives two solutions, for six solutions altogether.
While these equations can be solved exactly, one cannot do this in general.  A Gr\"obner basis is an equivalent
implicit representation of $V$ from which we may transparently  extract numerical invariants such as the number of
solutions or the dimension and degree of $V$.
Finer questions about individual solutions may require computing them numerically.

Numerical methods only compute numerical approximations of solutions to a system~\eqref{system}.
Thus  $(1.271+ .341\sqrt{-1}\,,\,-.500+.866\sqrt{-1})$ is an approximation of a solution to~\eqref{Eq:Ex2.1}.
A numerical approximation of a point $\bfz\in V$ is any point $\bfy\in\CC^n$ which is in some sense close to~$\bfz$.
For example, we could require that $\bfy$ is within some tolerance $\epsilon>0$ of $\bfz$, i.e., $|\bfy-\bfz|<\epsilon$.
Consequently, the concept of zero of (or solution to) a polynomial system is replaced by an \demph{approximate zero} (defined in
Section~\ref{sec:tracking}). 
This is fundamentally different than using exact methods like Gr\"obner bases, where the goal is to
work with a representation of the true exact zeros of polynomial systems.
As an approximate zero is not a true zero, a numerical computation does not yield all the information obtained in an exact
computation.
On the other hand, a numerical computation is often less costly than a symbolic computation.
Other advantages are  that the architecture of modern computers is optimized for floating point arithmetic and that numerical
continuation is readily parallelized (see Remark~\ref{R:parallel}).
Another advantage is that approximate zeroes are easily refined using Newton's Method (Section~\ref{sec:tracking}).

Despite not containing all the information of true zeros, in Section~\ref{sec:cert} we discuss how to use approximate zeros to obtain
precise and provable results.

%
\subsection{The Parameter Continuation Theorem}\label{SS:parameterContinuation}
Our discussion of polynomial homotopy continuation assumed that solution paths exist.
The \demph{Parameter Continuation Theorem} by Morgan and Sommese~\cite{MS1989} asserts this when the homotopy arises from
a path in \demph{parameter} space.

Suppose that the system of polynomials~\eqref{system} depends on parameters $\bfp=(p_1,\ldots,p_k)\in\CC^k$.
Write $F(\bfx;\bfp)$ for the 
polynomial system corresponding to a particular choice of $\bfp$, and further suppose that the map $\bfp\mapsto F(\bfx;\bfp)$ is algebraic
in that the coefficients of monomials in $F$ are polynomial or rational functions of $\bfp$.
For example, the parameters may be the coefficients in $F$.
Consider the incidence variety
 \begin{equation}\label{Eq:Incidence_I}
    Z\ \vcentcolon=\ \{ (\bfx, \bfp) \in \CC^n \times \CC^k \mid F(\bfx;\bfp) = 0\}\ \subseteq\ \CC^n \times \CC^k\,.
 \end{equation}
Let $\pi_1\colon Z \rightarrow \CC^n$ and  $\pi_2\colon Z \rightarrow \CC^k$ be the projections onto the first and second
factors.
The map $\pi_1$ identifies points in the fiber $\pi_2^{-1}(\bfp)$ with solutions to $F(\bfx;\bfp)=0$.
Standard arguments in algebraic geometry imply the existence of a dense open set $U\subset\CC^k$ and a number $d$ such that for all $\bfp\in U$, the fiber $\pi_2^{-1}(\bfp)$ has dimension $d$. 
Furthermore, when this typical dimension $d$ is zero, there is a number $N$ such that for $\bfp\in U$, then number of regular solutions in
the fiber $\pi_2^{-1}(\bfp)$ is $N$.


\begin{theorem}[Parameter Continuation  Theorem]\label{parameter_theorem}
 With these definitions and when $d=0$, suppose that  $\gamma(t)\colon [0, 1] \rightarrow \CC^k$ is a continuous path.
\begin{enumerate}[label=(\alph*)]
   \item If $\gamma([0, 1])\subset  U$, then the homotopy $H(\bfx,t)=F(\bfx; \gamma(t))$  defines $N$ continuous and isolated solution
     paths $\bfx(t)$.

   \item \label{param_thm_b}
      If $\gamma((0, 1])\subset  U$, then as $t \rightarrow 0$, the limits of the solution paths, if they exist, include all the isolated
     solutions to $F(\bfx;\gamma(0))=0$.
     This includes both regular solutions and solutions with multiplicity greater than one.
\end{enumerate}
At points $t\in[0,1]$ with $\gamma(t)\in U$ where $\gamma$ is differentiable, $\bfx(t)$ is differentiable.
\end{theorem}
A short proof for this theorem using Gr\"obner bases can be found in \cite{BB23}.

Strictly speaking, for differentiability at the endpoints, we want $\gamma$ to be defined on some neighborhood of $[0,1]\subset\RR$.
The  point of the theorem is that any path satisfying $\gamma((0,1])\subset U$ can be used for polynomial homotopy continuation, with
start system $G(\bfx)=F(\bfx;\gamma(1))$.
Since the \demph{branch locus} $B \vcentcolon= \CC^k{\smallsetminus}U$ is a subvariety, it has real codimension at least two and typical
paths in the parameter space $\CC^k$ do not meet $B$.
When the typical fiber dimension $d$ is zero, we call $\pi_2\colon Z\to\CC^k$ a \demph{branched cover}.
Theorem~\ref{parameter_theorem} can be generalized, replacing the parameter space $\CC^k$ by an irreducible
variety~\cite[Theorem 7.1.4]{Sommese:Wampler:2005}.
Similarly, polynomials in~$F$ may be replaced by Laurent polynomials, as we shall encounter in Section~\ref{polyhedral-homotopy}.
A \demph{parameter homotopy} is one arising from a path $\gamma$ such as in Theorem~\ref{parameter_theorem}\ref{param_thm_b}.

The Parameter Continuation Theorem follows from Bertini's Theorem, other standard results in algebraic
geometry, and the Implicit Function Theorem.
A proof is given in~\cite{MS1989}.

\begin{example}\label{ex:paths}
 Figure~\ref{F:paths} shows some of the possibilities for homotopy paths $\bfx(t)$, when
 Theorem~\ref{parameter_theorem}\ref{param_thm_b} holds. 
 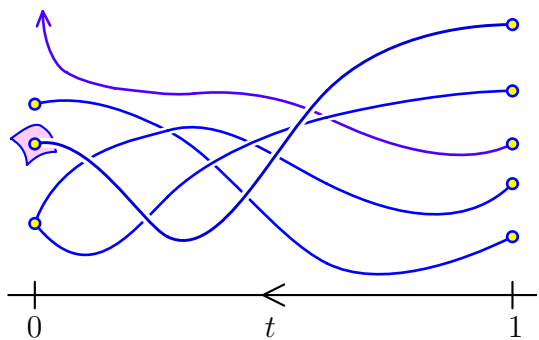
\begin{figure}[htb]

  \centering
  \begin{picture}(200,120)(0,-5)
    \put(0,0){\includegraphics{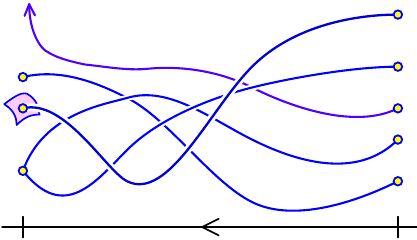}}
    \put(8,-10){$0$}   \put(98,-10){$t$}  \put(190,-10){$1$}
  \end{picture}

 \caption{Homotopy Paths.}\label{F:paths}
 \end{figure}
 The start system $G(\bfx)$ at $t=1$ has $N=5$ regular zeros, and each lies on a unique path $\bfx(t)$ for $t\in(0,1]$.
 One path has no finite limit as $t\to 0$, while the other four have limits.
 Two paths have the same limit, and their common endpoint is an isolated zero of~$F(\bfx)$ of multiplicity
 two.
 The endpoint of another path is not an isolated zero of $F(\bfx)$, indicated by a small pink surface.
 The endpoint of the other path is the regular zero of the target system $F(\bfx)$.  \hfill$\diamond$
\end{example}

\subsection{The total degree homotopy}\label{sec:totalDegree}

To reach {\em all} 
isolated zeros of the target system,
the start system must have at least as many zeros as the target system.
Thus, an upper bound on the number of isolated zeros is often needed to choose a homotopy.
One such upper bound is provided by B\'{e}zout's Theorem:
The number of isolated zeros of the system $F = (f_1,\ldots,f_n)$ is at most $d_1 d_2 \cdots d_n$,
where $d_i = \deg f_i$ for $i=1,\ldots,n$.
It inspires start systems of the form
 \begin{equation}\label{Eq:BezoutStart}
   \begin{bmatrix}
        \ b_0 x_1^{d_1}\ -\ b_1\ \\
        \vdots \\
        \ b_0 x_n^{d_n}\ -\ b_n\
    \end{bmatrix}\ .
 \end{equation}
For nonzero complex numbers $b_0,b_1,\ldots,b_n \in \CC^\times=\CC{\smallsetminus}\{0\}$, this start system is outside the branch locus $B$
and it has $d = d_1 d_2 \cdots d_n$ solutions, which are all are easily computed. 
This gives the \demph{total degree homotopy},
 \begin{equation}\label{E:TDH}
    H(x_1,\ldots,x_n,t) \vcentcolon=
    t \cdot
    \begin{bmatrix}
        \ b_0 x_1^{d_1}\ -\ b_1\ \\
        \vdots \\
        \ b_0 x_n^{d_n}\ -\ b_n\
    \end{bmatrix}
    \ +\
    (1-t) \cdot
    \begin{bmatrix}
        \ f_1(x_1,\ldots,x_n)\ \\
        \vdots \\
        \ f_n(x_1,\ldots,x_n)\
    \end{bmatrix}\ .
\end{equation}
Such a convex combination of two similar systems is called a \demph{straight-line homotopy}, which
is a particular case of a parameter homotopy.
For general choices of the parameters $b_i$ the smoothness conditions of Theorem~\ref{parameter_theorem}\ref{param_thm_b}
hold~\cite[Thm.~8.4.1]{Sommese:Wampler:2005}.

\begin{example}\label{Ex:TD_simpleSystem}
  The total degree homotopy for the system~\eqref{Eq:Ex2.1} has the form
    \begin{equation}\label{equ: total degree homotopy example}
        H(x,y,t) \vcentcolon=
        t\cdot
        \begin{bmatrix}
            \ b_0 x^{2} - b_1\ \\
            \ b_0 y^{3} - b_2\
        \end{bmatrix}
        +
        (1-t)\cdot
        \begin{bmatrix}
            \ x^2 + y^2 - 1\  \\
            \ x^2 - y^3 - y -1\
        \end{bmatrix}
        .
    \end{equation}
    For $b_0,b_1,b_2 \in \CC^\times$,  the start system
    $H(x,y,1) = (b_0 x^2 - b_1, b_0 y^3 - b_2)$ has six distinct complex zeros, all of which are regular,
    and the zero set of $H(x,y,t)$ consists of six paths
    in $\CC^2 \times [0,1]$, each smoothly parameterized by $t$ for {\em almost all} choices of $b_0,b_1,b_2$.
    The parameter $b_0$  is used to avoid cancellation of  the highest degree terms for all $t\in(0,1]$.\hfill$\diamond$
\end{example}

The phrase ``almost all'' in this example is because the set of choices of $b_i$ for which the paths are singular
lies in the complement of the dense open subset $U$ of Theorem~\ref{parameter_theorem}, and therefore 
has measure zero.
Such situations occur frequently in this field, often referred to as \demph{probability one}.

\hypertarget{path-tracking}{%
\subsection{Path-tracking}\label{sec:tracking}}

Path-tracking is the numerical core of polynomial homotopy continuation.

Suppose that $H(\bfx,t)$ for $\bfx\in\CC^n$ and $t\in\CC$ is a homotopy with target system
$F(\bfx)=H(\bfx,0)$ and start system $G(\bfx)=H(\bfx,1)$.
Further suppose that for $t\in (0,1]$, $H(\bfx,t)=0$ defines smooth paths $\bfx(t)\colon(0,1]\to\CC^n$ such that each
isolated solution to $F$ is connected to at least one regular solution to $G$ through some path, as in
Theorem~\ref{parameter_theorem}\ref{param_thm_b}. 
By the Implicit Function Theorem,  each isolated solution to $G$ is the endpoint~$\bfx(1)$ of a unique path $\bfx(t)$.
Lastly, we assume that all regular solutions to $G$ are known.

Given this, the isolated solutions to $F$ may be computed as follows:
For each regular solution $\bfx(1)$ to $G$, track the path $\bfx(t)$ from $t=1$ towards $t=0$.
If it converges, then $\bfx(0)=\lim_{t\to 0} \bfx(t)$ satisfies $F(\bfx(0))= 0$, and this will find all isolated solutions to $F$.
%
%

The path $\bfx(t)$ satisfies the Davidenko differential equation, and thus we may compute values $\bfx(t)$ by solving
the initial value problem~\eqref{Davidenko}.
Consequently, we may use any numerical scheme for solving initial value problems.
This is not satisfactory for solving nonlinear polynomial systems due to the propagation of error.

As the solution paths $\bfx(t)$ are defined implicitly, there are standard methods to mitigate error propagation.
Let $E$ be a system of $n$ polynomials in $n$ variables.
Given a point $\bfz_0$~  where the Jacobian matrix $JE=\frac{\partial E}{\partial \bfx}$ of $E$ is invertible, we may apply the
\demph{Newton operator}~\defcolor{$N_E$}  to~$\bfz_0$, obtaining $\bfz_1$,
\begin{equation}\label{Newton}
  \bfz_1\ \vcentcolon=\ N_E(\bfz_0)\ \vcentcolon=\ 
  \bfz_0 - \left( JE(\bfz_0)\right)^{-1} E(\bfz_0)\,.
\end{equation}
We explain this: if we approximate the graph of the function $E$ by its tangent plane at $(\bfz_0,E(\bfz_0))$, then
$\bfz_1\in\CC^n$ is the  unique zero of this linear approximation.
There exists a constant $0<c<1$ such that when $\bfz_0$ is sufficiently close to a regular zero $\bfz$ of $E$,
we have \demph{quadratic convergence} in that
\[
   \|\bfz_1-\bfz\|\ \leq\ c\|\bfz_0-\bfz\|^2\,.
\]
This is because $\bfz$ is a fixed point of $N_E$ at which the derivative of $N_E$ vanishes.
The inequality follows from standard error estimates from Taylor's Theorem for $N_E$ in a neighborhood of $\bfz$.
A consequence is that when $\bfz_0$ is sufficiently close to a regular zero $\bfz$, each Newton iterate starting from $\bfz_0$
doubles the number of accurate digits.
Such a point $\bfz_0$ is an \demph{approximate zero} of $F$.
This leads to algorithms to certify numerical output as explained in Section~\ref{sec:cert}.

\demph{Predictor-corrector algorithms} for solving the initial value problem for homotopy paths
take a discretization $1=t_0>t_1>\dotsb>t_m=0$ of the interval $[0,1]$ and iteratively compute
approximations $\bfx(1)=\bfx(t_0),\bfx(t_1),\dotsc,\bfx(t_m)=\bfx(0)$  to points on the solution path $\bfx(t)$.
This requires an initial approximation $\bfx_0$ to $\bfx(t_0) = \bfx(1)$.
Then, for $k=0,\dotsc,m{-}1$, given an approximation $\bfx_k$ to $\bfx(t_k)$, a prediction $\hat{\bfx}_{k+1}$ for $\bfx(t_{k+1})$ is
computed.
This typically uses one step in an iterative method for solving the initial value problem (a local solver).
This is the \demph{predictor step}. 
Next, one or more Newton iterations $N_E$ for $E(\bfx)=H(\bfx,t_{k+1})$ are applied to $\hat{\bfx}_{k+1}$ to obtain a new
approximation $\bfx_{k+1}$ to $\bfx(t_{k+1})$.
This is the \demph{corrector step}.
Predictor steps generally cause us to leave the proximity of the path being tracked; corrector steps bring us back.
The process repeats until $k=m{-}1$.

There are a number of efficient local solvers for solving initial value problems.
They typically use approximations to the Taylor series for the trajectory $\bfx(t)$ for $t$ near $t_k$.
For example, the Euler predictor uses the tangent line approximation,
\[\
   \hat{\bfx}_{k+1}\ \vcentcolon=\ \bfx_k+\Delta t_k \Delta \bfx_k
   \qquad\mbox{where}\qquad
   \frac{\partial H}{\partial \bfx}(\bfx_k,t_k) \cdot \Delta \bfx_k + \frac{\partial H}{\partial t}(\bfx_k,t_k)\ =\ 0\,.
\]
Here, $\Delta t_k=t_{k+1}-t_k$ and $(\Delta x_k,1)$ spans the kernel of the Jacobian $JH(\bfx_k,t_k)$.

Figure~\ref{Fig:Predictor-Corrector} illustrates an Euler prediction followed by Newton corrections.
\begin{figure}[htb]
  \centering
  \begin{picture}(330,166)(0,-6)
  \put(0,0){\includegraphics{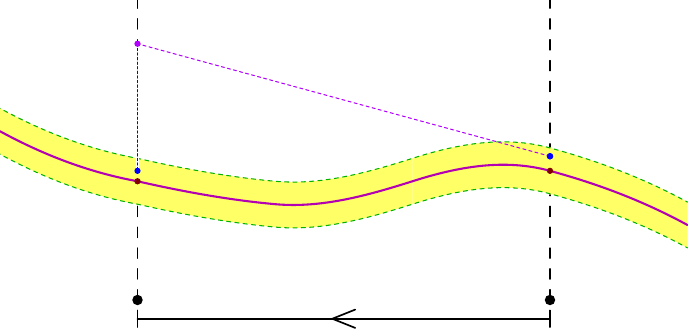}}

   \put(155,142){predictor} \put(157,137){\vector(-1,-1){15}} \put(40,138){{\color{Plum}$\hat{\bfx}_{k+1}$}}
   \put(91,104){corrector} \put(89,106.5){\vector(-1,0){21}}

   \put(96,85){{\color{blue}$\bfx_{k+1}$}}\put(94,87.5){{\color{blue}\vector(-3,-1){25}}}
    \put(284,109){{\color{blue}$\bfx_{k}$}}\put(285,105.5){{\color{blue}\vector(-1,-1){16}}}

  \put(6,40){{\color{mymaroon}$\bfx(t_{k+1})$}} \put(42,49){{\color{mymaroon}\vector(1,1){22}}}


  \put(214,45){{\color{mymaroon}$\bfx(t_{k})$}}   \put(240.7,54){{\color{mymaroon}\vector(1,1){22}}}

  \put(100,24){$\epsilon$-neighborhood} \put(135,34){\vector(0,1){15}}
  \put(43,13){$t_{k+1}$} \put(157,-9){$\Delta t_k$} \put(270,13){$t_{k}$}
  \end{picture}
  \caption{Euler prediction followed by Newton corrections.
    The image is adapted from \cite{BT_Intro} (we thank Sascha Timme for allowing us to use his figure).}
 \label{Fig:Predictor-Corrector}
\end{figure}
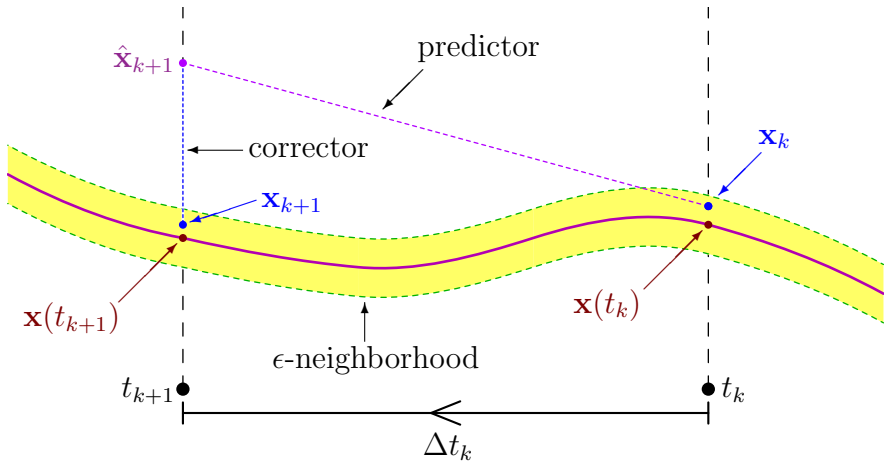
It suggests a stopping criterion for Newton iterations based on a fixed tolerance $\epsilon$.
Another is to apply Newton iterations until quadratic convergence is observed.

\begin{remark}\label{R:parallel}
  Since each solution path defined by $H(\bfx,t) = 0$ may be tracked independently,
  path-tracking is (in the words of Jan Verschelde) pleasingly parallelizable~\cite{VerscheldeYoffe2010Polynomial},
  which is a strength of polynomial homotopy continuation.\hfill$\diamond$  
\end{remark}

We have hardly mentioned 
the state of the art in path-tracking methods.
There is a significant literature on other predictor-corrector schemes (see \cite[Sections 15--18]{BC2013} for an overview), practical 
path-tracking heuristics~\cite{BertiniBook,Timme21}, and endgames for dealing with issues that arise near $t=0$,
such as divergent~\cite{Morgan:1986} or singular paths~\cite{Morgan:Sommese:Wampler:1992a}.
Indeed, the methods we describe would suffice only for the most 
well-conditioned paths ending at regular solutions at $t=0$.

\subsection{Squaring up}\label{SS:SquaringUp}
It is common to need to solve \demph{overdetermined systems}, which have more equations than variables.
While the version of the Parameter Continuation Theorem we presented holds for overdetermined systems, 
this presents a challenge as both the total degree homotopy of Section~\ref{sec:totalDegree} and the polyhedral homotopy from the next
section enable us to find all isolated solutions to a {\it square} system $F(\bfx)=0$ of polynomial equations.
Let us discuss an example, and then explain the method of squaring up, which reduces the problem of solving overdetermined systems to that
of solving square systems.

\begin{example}\label{Ex:squareUp}
  Let $A,B,C$ be the following $2\times 3$ matrices,
\[
  A\ \vcentcolon=\ \left(\begin{matrix}1&3&5\\2&4&6\end{matrix}\right)\ ,
    \qquad
  B\ \vcentcolon=\ \left(\begin{matrix}2&3&7\\2&5&-11\end{matrix}\right)\ ,
    \quad\mbox{ and }\quad 
  C\ \vcentcolon=\ \left(\begin{matrix}1&-1&1\\-2&3&-7\end{matrix}\right)\ .
\]
We ask for the matrices of the form $D(x,y)\vcentcolon=A+Bx+Cy$ that have rank 1.
This is given by the vanishing of the three $2\times 2$ minors of $D(x,y)$.
This overdetermined system of equations in $x,y$ has three solutions:
 \begin{equation}\label{ThreePoints}
    (-\tfrac{4}{5}, \tfrac{3}{5})\,,\  (-0.15019, 0.16729)\,,\,  (-0.95120, 2.8373)\,.
 \end{equation}
We may find these using the total degree homotopy as follows.
The determinants of the first two columns and the last two columns of $D(x,y)$ give a square subsystem of the system of three minors,
and these have $4=2\cdot 2$ solutions (the B\'ezout number).
In addition to the three solutions in~\eqref{ThreePoints}, the fourth is $(-13/14, 3/14)$.
Let $f(x,y)$ be the remaining minor. 
Then $f(-13/14, 3/14)=-963/98$, while the three solutions in~\eqref{ThreePoints} evaluate (close to) zero.
This is a simplification of the general scheme.  \hfill$\diamond$
\end{example}

Let $F$ be an overdetermined system consisting of $m$ polynomials in $n$ variables, where $m>n$.
\demph{Squaring up} $F$ replaces it by a square system $G(\bfx)\vcentcolon=M F(\bfx)$ as follows:
Let $M$ be a (randomly chosen) $n\times m$ complex matrix, so that $G(\bfx)$ consists of $n$ polynomials, each of which is a linear combination of
polynomials in $F(\bfx)$.
Next, find all isolated solutions to $G(\bfx)=0$.
Since the solutions of $F(\bfx)=0$ are among those of $G(\bfx)=0$, we need only to determine the zeros of $G$ which are not zeros of
$F$.
A simple way is to evaluate $F$ at each of the zeros of $G$ and discard those that do not evaluate to zero (according to some heuristic).
It is numerically more stable to apply the Newton operator for the  overdetermined system $F$ \cite{DS2000} to the zeros of $G$ and retain
those which converge quadratically.
Example~\ref{Ex:squareUp} is a simplification, using a very specific matrix $M$ rather than a randomly chosen matrix.  

\begin{remark}
Suppose that the overdetermined system $F(\bfx)$ depends on a parameter $\bfp\in\mathbb C^k$, i.e., we have
$F(\bfx)=F(\bfx;\bfp)$, and that for a general parameter $\bfp\in\mathbb C^k$ the system of equations $F(\bfx;\bfp)$ has $N>0$ isolated
solutions (as in the Parameter Continuation Theorem~\ref{parameter_theorem}). Suppose further that we have already computed all the
solutions of $F(\bfx;\bfp_0)=0$ for a fixed parameter $\bfp_0\in\mathbb C^k$ (either by squaring up or by using another method).
Then, we can use the Newton operator for overdetermined systems from \cite{DS2000} for polynomial homotopy continuation along the path~$F(\bfx;t\bfp_0 +(1-t)\bfp)$ for any other parameter $\bfp$.\hfill$\diamond$ 
\end{remark}

\section{Polyhedral homotopy}
\label{polyhedral-homotopy}

The total degree homotopy from \Cref{sec:totalDegree} may be used to compute all isolated zeros to any system of polynomial equations.
Its main flaw is that it is based on the B\'ezout bound.
Many polynomial systems arising in nature have B\'ezout bound  dramatically larger than their  number of zeros; for these,
the total degree homotopy will track many excess paths.

\begin{example}\label{Ex:biquadratic}
  To illustrate this potential inefficiency,
  consider the following problem, posed in~\cite{DHOST2016}:
  Find the distance from a point~$\bfx^*\in\RR^d$ to a hypersurface given by the vanishing of a single polynomial $f$.
  A first step is to compute all critical points of the distance function $\|\bfx-\bfx^*\|$ for $f(\bfx)=0$.
  We formulate this using a Lagrange multiplier $\lambda$,
  \[
  f(\bfx)\ =\ 0\ \quad\mbox{ and }\quad
  \lambda(\bfx-\bfx^*)\ =\ \nabla f(\bfx)\,.
  \]
  When $d=2$, $f=5-3x_2^2-3x_1^2+x_1^2x_2^2$, and $x^*=(0.025,0.2)$, these equations become
  \begin{equation}\label{Eq:biquadratic}
    5-3x_2^2-3x_1^2+x_1^2x_2^2\ =\ 0\quad\mbox{ and }\quad
    \lambda\left[\begin{array}{c} x_1-0.025\\x_2-0.2\end{array}\right]
      \ =\ \left[\begin{array}{c}
        -6x_1+2x_1x_2^2\vspace{2pt}\\-6x_2+2x_1^2x_2\end{array}\right]\ ,
  \end{equation}
  which are polynomials in $x_1,x_2,\lambda$ of degrees $4,3,3$, respectively.
  The system \eqref{Eq:biquadratic} has 12 solutions.
  We show the corresponding critical points below.
  \[
  \includegraphics[height=100pt]{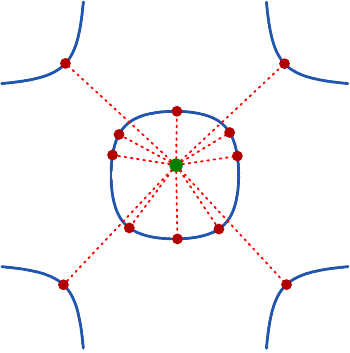}
  \]
  Note that a total degree homotopy for solving \eqref{Eq:biquadratic} follows $36>12$  homotopy paths. \hfill{$\diamond$}
\end{example}

The discrepancy between the B\'ezout bound of 36 and the actual number of 12 solutions in this example is not due to the coefficients
  in the polynomials, but rather the monomials which appear in them.
  A polynomial of degree $d$ in which not all terms of degree at most $d$ occur is called \demph{sparse}.

An alternative general-purpose homotopy is the \demph{polyhedral homotopy} of Sturmfels and Huber~\cite{HuSt95} for solving systems of
  sparse polynomials.
It is based on \demph{Bernstein's bound}~\cite[Thm.~A]{Bernshtein1975Number}, which is at most the B\'ezout bound and at least the
actual number of isolated zeros. 
Bernstein's bound is often significantly smaller than the B\'ezout bound, which makes the polyhedral homotopy an efficient tool for polynomial
homotopy continuation because it produces fewer excess paths to track.
Oftentimes it is optimal in that there are no excess paths to track~\cite{EDDMV,LindbergNicholsonRodriguez}.

The \demph{Polyhedral Homotopy Algorithm} is summarized in Algorithm \ref{polyhedral_algo} below.
It is implemented in \texttt{PHCpack} \cite{PHCpack}, the \texttt{HOM4PS} family \cite{HOM4PSArticle,LeeLiTsai2008HOM4PS}, and
\texttt{HomotopyContinuation.jl}~\cite{HC.jl}.
To understand how it works we first develop some theory.
We begin with Bernstein's bound.

\subsection{Bernstein's bound}\label{S:Bernstein}

The polyhedral homotopy takes place on the complex torus~$(\CC^\times)^n$, where
$\defcolor{\CC^\times}\vcentcolon=\CC\smallsetminus\{0\}$ is the set of invertible complex numbers.
Each integer vector $\bfa\in\ZZ^n$ gives a \demph{Laurent monomial} $\defcolor{\bfx^\bfa}\vcentcolon=x_1^{a_1}\dotsb x_n^{a_n}$, which is a
function on the torus $(\CC^\times)^n$. 
A linear combination of Laurent monomials,
\[
    f\ \vcentcolon=\ \sum_{\bfa\in\calA} c_{\bfa} \bfx^\bfa\qquad\quad c_\bfa\in\CC^\times\,,
\]
is a \demph{Laurent polynomial}.
The (finite) index set $\defcolor{\calA}\subset\ZZ^n$ is the \demph{support} of $f$.
The convex hull of $\calA$ is the \demph{Newton polytope} of $f$.
The support of the
polynomial $f$ in Example~\ref{Ex:biquadratic}
is the set containing the columns of the matrix
$\left(\begin{smallmatrix} 0&0&2&2\\0&2&0&2\end{smallmatrix}\right)$,
and its Newton polytope is the $2\times 2$ square, $[0,2]\times[0,2]$.

Bernstein's bound concerns square systems of Laurent polynomials, and it is in terms of mixed volume~\cite[pp.~116--118]{Ewald}.
The \demph{Minkowski sum} of polytopes $P$ and $Q$ in $\RR^n$ is
 \[
\defcolor{P+Q}\ \vcentcolon=\ \{\bfx + \bfy \mid \bfx\in P\mbox{ and }\bfy\in Q\}\,.
 \]
Given polytopes $P_1,\dotsc,P_n$ in $\RR^n$ and positive scalars $t_1,\dotsc,t_n$, Minkowski proved that the volume
$\vol( t_1 P_1 + \dotsb + t_n P_n)$ 
is a homogeneous polynomial in $t_1,\dotsc,t_n$ of degree $n$.
He defined the \demph{mixed volume}
\defcolor{$\MV(P_1,\dotsc,P_n)$} to be the coefficient of $t_1\dotsb t_n$ in that polynomial.
While mixed volume is in general hard to compute, when $n=2$, we have the formula
\begin{equation}\label{polarization}
   \MV(P,Q)\ =\ \vol(P+Q)-\vol(P)-\vol(Q)\,.
\end{equation}
This formula,
reminiscent of the inclusion-exclusion formula,
and its generalizations to $n> 2$ are the \demph{polarization identities}.

Consider \eqref{polarization} for the $2\times 2$ square and triangle below.
 \begin{equation}\label{Eq:MinkowskiSum}
   P\ =\ \raisebox{-15pt}{\includegraphics{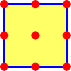}}
   \qquad
   Q\ =\ \raisebox{-15pt}{\includegraphics{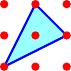}}
   \qquad
   P+Q\ =\ \raisebox{-30pt}{\includegraphics{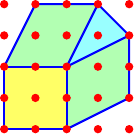}}
 \end{equation}
The final shape is the Minkowski sum $P+Q$, and the mixed volume $MV(P,Q)$ is the sum of the areas of the two (non-square) parallelograms,
which is 8. 

We give a version of Bernstein's Theorem~\cite[Thm.~A]{Bernshtein1975Number}.

\begin{theorem}[Bernstein]\label{thm: Bernstein}
  Let $f_1,\dotsc,f_n$ be Laurent polynomials
  and $P_1,\ldots,P_n$ be their respective Newton polytopes.
  The number of isolated solutions to $F=(f_1,\ldots,f_n)$ in $(\CC^\times)^n$ is at most $\MV(P_1,\dotsc,P_n)$.
\end{theorem}

\begin{example}\label{Ex:more EDD}
  We show the supports and Newton polytopes of the three polynomials from \Cref{Ex:biquadratic},
  $f$, $\lambda(x_1-x_1^*)-\partial f/\partial x_1$, and  $\lambda(x_2-x_2^*)-\partial f/\partial x_2$ (see Equation~\eqref{Eq:biquadratic}).
 \[
   \begin{picture}(95,90)(-7,0)
     \put(0,0){\includegraphics{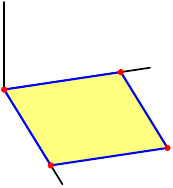}}
     \put(-7,80){\small$\lambda$}  \put(15,2){\small$x_1$}  \put(61,62){\small$x_2$}
   \end{picture}
     \qquad
   \begin{picture}(95,90)(-7,0)
     \put(0,0){\includegraphics{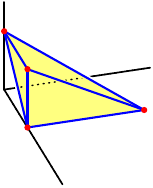}}
     \put(-7,80){\small$\lambda$}  \put(15,2){\small$x_1$}  \put(61,62){\small$x_2$}
   \end{picture}
    \qquad
   \begin{picture}(95,90)(-7,0)
     \put(0,0){\includegraphics{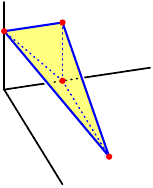}}
     \put(-7,80){\small$\lambda$}  \put(15,2){\small$x_1$}  \put(61,62){\small$x_2$}
   \end{picture}
   \]
   Their mixed volume is twelve, so the system~\eqref{Eq:biquadratic} achieves the Bernstein bound.  \hfill$\diamond$
\end{example}

Bernstein proved that this bound is typically achieved in the following sense:
For each $i=1,\dotsc,n$, let \defcolor{$\calA_i$} be the support of polynomial $f_i$ and \defcolor{$P_i$} its Newton polytope.
The set of polynomial systems $G=(g_1,\dotsc,g_n)$ where each $g_i$ has support a subset of $\calA_i$ is a vector space \defcolor{$V$} of
dimension $|\calA_1|+\dotsb+|\calA_n|$ whose coordinates are given by the coefficients of the polynomials $g_i$.
Bernstein showed that there is a nonempty Zariski open subset $U\subset V$ consisting of systems $G$ with exactly $\MV(P_1,\dotsc,P_n)$ regular
zeros.
Also, if \defcolor{$U'$} is the (larger) set of systems $G$ with exactly $\MV(P_1,\dotsc,P_n)$ solutions, counted with multiplicity, then
$U'$ is open, and Bernstein~\cite[Thm.~B]{Bernshtein1975Number} gave an effective criterion for when $G\in V\smallsetminus U'$.

We remark that Bernstein's bound and Bernstein's Theorem are often called the Bernstein-Kushnirenko-Khovanskii (BKK) bound and BKK Theorem
due to their joint paper~\cite{BKK} and the circle of closely related
work~\cite{Bernshtein1975Number,Kushnirenko1976Polyedres,Khovanskii1978Newton}.
%
%

\subsection{Polyhedral homotopy of Huber and Sturmfels}

In their seminal work~\cite{HuSt95}, Huber and Sturmfels developed a homotopy continuation method for solving systems of Laurent polynomials
that takes advantage of the polyhedral structure of the system and Bernstein's bound in that it tracks only $\MV(P_1,\dotsc,P_n)$ paths.
Their work also provided a new interpretation for mixed volume in terms of mixed cells
(the parallelograms in~\eqref{Eq:MinkowskiSum}) and a new proof of Bernstein's Theorem~\cite[Thm.~A]{Bernshtein1975Number}.
We sketch its main ideas in some detail.
This is also described in Sturmfels' award-winning article in the American Mathematical Monthly~\cite{StMon}.

\begin{remark}
 It is worth noting that around the same time,
 Verschelde, Verlinden, and Cools independently developed an alternative homotopy for solving systems of sparse polynomials
 that takes advantage of their combinatorial structure and Bernstein's bound \cite{VVC1994}.\hfill$\diamond$
\end{remark}

%
%

Suppose that $F=(f_1,\dotsc,f_n)$ is a system of Laurent polynomials and that $\calA_i$ is the support of $f_i$ for $i=1,\dotsc,n$, so that
 \begin{equation}\label{Eq:target}
   f_i\ =\ \sum_{\bfa\in\calA_i} c_{i,\bfa} \bfx^\bfa\,,\quad\mbox{with }c_{i,\bfa}\in \CC^\times\, .
 \end{equation}
For now, assume that $F$ is sufficiently generic, in that $F$ has $\MV(P_1,\dotsc,P_n)$ regular zeros.

In the polyhedral homotopy, the continuation parameter $t$ appears in a very different way than in the total degree homotopy~\eqref{E:TDH}.
First, the start system is at $t=0$ and the target system $F$ is at $t=1$, but this is not a substantive difference.
The polyhedral homotopy depends upon a choice of \demph{lifting functions}, $\defcolor{\omega_i}\colon\calA_i\to\ZZ$, for $i=1,\dotsc,n$.
That is, a choice of an integer $\omega_i(\bfa)$ for each monomial $\bfa$ in $\calA_i$.
We address this choice in \Cref{sec: mixed cells}.

Given lifting functions, define the homotopy $H(\bfx,t)\vcentcolon=(h_1,\dotsc,h_n)$ by
\[
  h_i(\bfx,t)\ \vcentcolon=\
   \sum_{\bfa\in\calA_i} c_{i,\bfa}\, \bfx^{\bfa}\, t^{\omega_i(\bfa)}\,.
\]
By the Parameter Continuation Theorem (Theorem~\ref{parameter_theorem}) and our genericity assumption on $F$,
over $t\in(0,1]$ the system of equations $H(\bfx,t)=0$ defines
  $\MV(P_1,\dotsc,P_n)$ smooth paths. 
It is however not at all clear what happens as $t\to 0$.
For example, $H(\bfx,0)$ is undefined if some $\omega_i(\bfa)<0$, and if $\omega_i(\bfa)>0$ for all $i$ and $\bfa$, then $H(\bfx,0)$ is
identically zero.

The key idea is to use an invertible change of coordinates to study the homotopy paths as $t\to 0$.
This coordinate change depends upon a \demph{weight} $\defcolor{\bfalpha}\in\ZZ^n$ and a positive integer~\defcolor{$r$}.
Together, they describe the order of a series
expansion of a subset of homotopy paths near $t=0$.
The weight gives a \demph{cocharacter} of the torus, for $s\in\CC^\times$, $\defcolor{s^{\bfalpha}}\vcentcolon=(s^{\alpha_1},\dotsc,s^{\alpha_n})$.
Set
\[
   \defcolor{\bfy}\ =\ \bfx\circ s^{-\bfalpha} \  \vcentcolon=\ (x_1 s^{-\alpha_1}, \dotsc, x_n s^{-\alpha_n})\,.
\]
Then $\bfx=\bfy\circ s^\bfalpha$, and we define
$H^{(\bfalpha)}(\bfy,s)\vcentcolon=(h_1^{(\bfalpha)},\dotsc,h_n^{(\bfalpha)})$,
where for $i=1,\dotsc,n$,
 \begin{equation}\label{Eq:alpha-system}
   \defcolor{h_i^{(\bfalpha)}(\bfy,s)}\ \vcentcolon=\ s^{-\beta_i} h_i(\bfy\circ s^{\bfalpha}, s^r)\ =\
   \sum_{\bfa\in\calA_i} c_{i,\bfa}\, \bfy^{\bfa}\, s^{\langle \bfalpha, \bfa\rangle + r\omega_i(\bfa)-\beta_i}\,,
 \end{equation}
where $\defcolor{\beta_i}\vcentcolon=\min\{ \langle \bfalpha, \bfa\rangle + r\omega_i(\bfa)\mid \bfa\in\calA_i\}$.
The purpose of $\beta_i$ is to ensure that $s$ appears in $h_i^{(\bfalpha)}(\bfy,s)$ with only non-negative exponents, and that
$h_i^{(\bfalpha)}(\bfy,0)$ is defined and not identically zero.
Specifically, if
$\defcolor{\calA_i^{(\bfalpha)}}\vcentcolon=\{\bfa\in\calA_i\mid  \langle \bfalpha, \bfa\rangle + r\omega_i(\bfa)=\beta_i\}$, then
 \begin{equation}\label{Eq:alpha-facial-system}
   h_i^{(\bfalpha)}(\bfy,0)\ =\
   \sum_{\bfa\in\calA_i^{(\bfalpha)}} c_{i,\bfa} \bfy^\bfa
   \qquad\mbox{for } i=1,\dotsc,n\,.
 \end{equation}
The purpose of the positive integer $r$ is to keep the exponents integral.
As $r>0$, we have that for $s\in[0,1]$,  $t=s^r\to 0$ if and only if $s\to 0$.
Thus the role of $r$ and $s$ is to parameterize the homotopy path
through a series expansion.
We remark on this later
in the sketch of the proof for Algorithm \ref{polyhedral_algo}.

We will see that for almost all $(\bfalpha,r)$, the system $H^{(\bfalpha)}(\bfy,0)$ has no zeros in $(\CC^*)^n$,
but for appropriately chosen $(\bfalpha,r)$,
the system~$H^{(\bfalpha)}(\bfy,0)$ has easily computed zeros, each defining a homotopy path to a solution of
$H^{(\bfalpha)}(\bfy,1)=H(\bfx,1)$.
For such an $(\bfalpha,r)$, $H^{(\bfalpha)}(\bfy,0)$ is a \demph{start subsystem}.
The polyhedral homotopy algorithm consists of determining those $(\bfalpha,r)$, solving 
the start subsystems $H^{(\bfalpha)}(\bfy,0)$, and then
tracking the homotopy paths from $t=0$ to $t=1$.

Before discussing this in more detail, including the role of the choices of lifting functions~$\omega_i$, cocharacter~$\bfalpha$, and
positive integer $r$, let us consider an example.

\begin{example}\label{Ex:phtpy}
  Let $f_1$ be the biquadratic from Example~\ref{Ex:biquadratic} and let
  $f_2= 1 + 2x_1x_2 - 5x_1x_2^2 - 3x_1^2 x_2$.
  Here, $\calA_1=\left(\begin{smallmatrix} 0&0&2&2\\0&2&0&2\end{smallmatrix}\right)$ and
  $\calA_2=\left(\begin{smallmatrix} 0&1&1&2\\0&1&2&1\end{smallmatrix}\right)$.
  Their Newton polygons are the square $P$ and triangle $Q$ in~\eqref{Eq:MinkowskiSum}.
  Figure~\ref{F:Eight_Sols} shows their $8=\MV(P,Q)$ common zeros.
  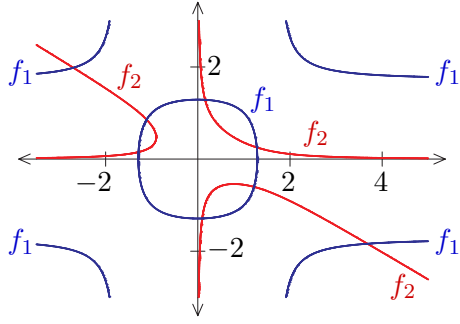
\begin{figure}[htb]
  \centering
   \begin{picture}(169,120)(-3.5,0)
     \put(0,0){\includegraphics[height=120pt]{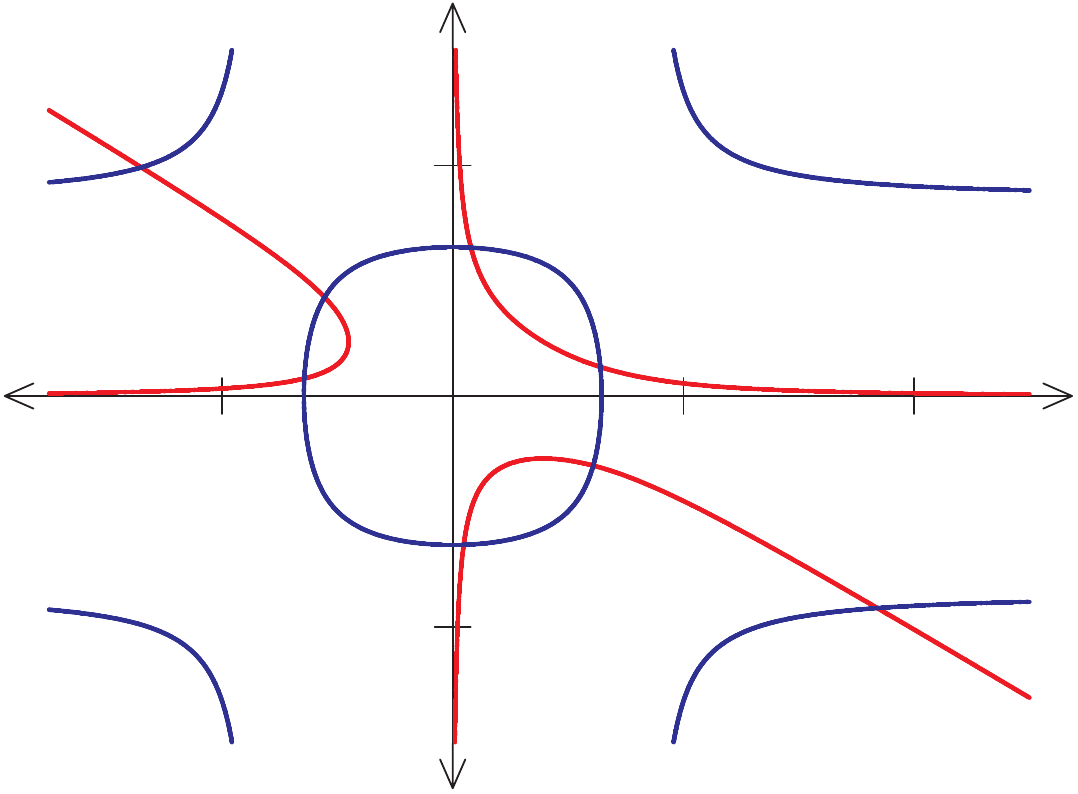}}
     \put( 37.5, 88.5){{\color{myred}\small$f_2$}}     \put(108, 67){{\color{myred}\small$f_2$}}
     \put(141,  9){{\color{myred}\small$f_2$}}
     \put(-3.5,90){{\color{myblue}\small$f_1$}}     \put(158, 90){{\color{myblue}\small$f_1$}}
         \put(88,80){{\color{myblue}\small$f_1$}}
     \put(-3.5,26){{\color{myblue}\small$f_1$}}     \put(158,26){{\color{myblue}\small$f_1$}}
   \put(22,48){\small$-2$}    \put(100,48){\small$2$}   \put(135,48){\small$4$}
    \put(72,23){\small$-2$}    \put(72,92){\small$2$}
  \end{picture}
  \caption{
    The zero sets of $f_1$ (blue) and $f_2$ (red) together with their intersections}.\label{F:Eight_Sols}
  \end{figure}

Define $\omega_1$ to be identically zero and set $\omega_2\left(\begin{smallmatrix}a\\b\end{smallmatrix}\right)\vcentcolon=a+b$.
Then $h_1(\bfx,t)=f_1(\bfx)$ and
\[
h_2(\bfx,t)\ =\ 1 + 2 x_1 x_2 t^2 - 5x_1x_2^2 t^3 - 3x_1^2 x_2 t^3\,.
\]
Let $\bfalpha=(0,-3)$ and $r=2$, so that $x_1=y_1$ and $x_2=y_2 t^{-3}$.
Then we may check that $\beta_1=-6$ and $\beta_2=\min\{0-0,4-3,6-6,6-3\}=0$, so that
\begin{eqnarray*}
  h_1^{(\bfalpha)}(\bfy,s) &=& \underline{- 3 y_2^2 + y_1^2y_2^2}\ +\ s^6(5 - 3y_1^2)\,,\\
  h_2^{(\bfalpha)}(\bfy,s) &=& \underline{1 - 5 y_1y_2^2}\ +\ s(2y_1y_2  - 3y_1^2 y_2 s^2)\,.
\end{eqnarray*}
The underlined terms are the binomials $h_1^{(\bfalpha)}(\bfy,0)$ and $h_2^{(\bfalpha)}(\bfy,0)$. 
They have four solutions,
\[
   ( \sqrt{3}, 75^{-\frac14})\,,\ 
   ( \sqrt{3},-75^{-\frac14})\,,\ 
   (-\sqrt{3}, 75^{-\frac14})\,,\ 
   (-\sqrt{3},-75^{-\frac14})\,. 
\]
These four solutions lead to four homotopy paths, which is fewer than the $8=\MV(P,Q)$ paths defined by $H(\bfx,t)=0$ over
$t\in(0,1]$.

If we let $\bfgamma=(-3,0)$ and $r=2$, then
\begin{eqnarray*}
  h_1^{(\bfgamma)}(\bfy,s) &=& - 3 y_1^2 + y_1^2y_2^2\ +\ s^6(5 - 3y_2^2)\,,\\
  h_2^{(\bfgamma)}(\bfy,s) &=& 1 - 3y_1^2 y_2 \ +\ s(2y_1y_2 - 5 y_1y_2^2 s^2)\,,
\end{eqnarray*}
so that $h_1^{(\bfgamma)}(\bfy,0)$ and $h_2^{(\bfgamma)}(\bfy,0)$ are again binomials and they have four solutions
\[
   ( 27^{-\frac14}, \sqrt{3})\,,\ 
   (-27^{-\frac14}, \sqrt{3})\,,\ 
   ( 27^{-\frac14},-\sqrt{3})\,,\ 
   (-27^{-\frac14},-\sqrt{3})\,.
\]
These lead to the other four homotopy paths.
The partition $4+4=8$ is seen in the decomposition of $P+Q$ in~\eqref{Eq:MinkowskiSum}.
The weight $\bfalpha$ corresponds to the parallelogram on the upper left, which is the Minkowski sum of the supports of
the components $h_1^{(\bfalpha)}(\bfy,0)$ and $h_2^{(\bfalpha)}(\bfy,0)$ of the start subsystem, and the weight $\bfgamma$ corresponds to
the parallelogram on the lower right. 
The only weights and positive integers for which the start subsystem has solutions are positive multiples of $(\bfalpha,2)$ and
$(\bfgamma,2)$.\hfill$\diamond$ 
\end{example}

\subsection{Mixed cells}\label{sec: mixed cells}
In \Cref{Ex:phtpy} only two choices of $(\bfalpha, r)$ gave start subsystems  $H^{(\bfalpha)}(\bfy,0)$ with solutions.
We now address the problem of computing the pairs $(\bfalpha, r)$, such that $H^{(\bfalpha)}(\bfy,0)$ has solutions.
This leads to an algorithm that computes these pairs given the start system $F=(f_1,\ldots,f_n)$.
We will show that the pairs $(\bfalpha, r)$ which give zeros at $t=0$ correspond to certain \demph{mixed cells}
in a decomposition of the Minkowski sum $P_1+\cdots+P_n$, where $P_i$ is the Newton polytope of $f_i$.

We examine the geometric combinatorics of the lifting functions $\omega_i$,  weights $\bfalpha$, and positive integer $r$.
Let $P\subset\RR^{n+1}$ be a polytope.
If $P$ has the same dimension as its projection to $\RR^n$, then it and all of its faces are \demph{lower faces}.
Otherwise, replace $\RR^{n+1}$ by the affine span of $P$ and assume that $P$ has dimension $n+1$.
A \demph{lower facet} of $P$ is a \demph{facet} $Q$ of $P$ ($\dim Q=n$) whose inward-pointing normal vector has positive last coordinate.
A \demph{lower face} of $P$ is any face lying in a lower facet.
The union of lower faces forms the \demph{lower hull} of $P$.

Let $\calA\subset\ZZ^n$ be a finite set and $\omega\colon\calA\to\ZZ$ be a lifting function.
The \demph{lift} of $\calA$ is the set
\[
\defcolor{\widehat{\calA}}\ \vcentcolon=\ \{(\bfa,\omega(\bfa))\mid\bfa\in\calA\}\subset\ZZ^{n+1}\,.
\]
Let $\defcolor{\widehat{P}}\vcentcolon=\conv(\widehat{\calA})$ be its convex hull.
Given a lower face $Q$ of $\widehat{P}$, the projection to $\ZZ^n$ of $Q\cap\widehat{\calA}$ is a subset $\defcolor{\calC(Q)}$ of $\calA$
whose convex hull is the projection to $\RR^n$ of $Q$.
If $(\bfalpha,r)$ is upward-pointing ($r>0$) and $\bfa \mapsto \langle\bfalpha,\bfa\rangle + r\omega(\bfa)$ achieves its minimum on
$Q$, then we may observe that $\calC(Q)= \calA^{(\bfalpha)}$.

For each $i=1,\dotsc,n$, let $\calA_i\subset\ZZ^n$ be a finite set, $\omega_i\colon\calA_i\to\ZZ$ be a lifting function, and set
$\widehat{P}_i\vcentcolon=\conv(\widehat{\calA}_i)$.
Let $\defcolor{\widehat{P}}\vcentcolon=\widehat{P}_1+\dotsb+\widehat{P}_n$ be their Minkowski sum.
As $\widehat{P}$ is a Minkowski sum, if $Q$ is a lower face of $\widehat{P}$, for each $i=1,\dotsc,n$ there is a lower face
$Q_i$ of $\widehat{P}_i$ with 
 \begin{equation}
   \label{Eq:lowerSum}
   Q\ =\ Q_1+\dotsb+Q_n\,.
 \end{equation}
%

 \begin{definition}\label{def_generic_lifting_function}
   Lifting functions $\omega_i\colon\calA_i\to\ZZ$ for $i=1,\dotsc,n$ are
   \demph{generic} if for each lower facet
   $Q = Q_1 + \cdots + Q_n$ of $\widehat{P}$, with $Q_i$ being a lower face of $\widehat{P}_i$ for $i=1,\ldots,n$, we have that
   \begin{equation}
     \label{Eq:dimSum}
   \dim Q\ =\ n\ =\ \dim Q_1 + \dotsb + \dim Q_n\,,
 \end{equation}
and $\dim Q_i = \# (Q_i\cap\widehat{\calA}_i) - 1$.

A lower facet $Q$ for which every $Q_i$ in~\eqref{Eq:dimSum} has dimension 1 (and thus $\# Q_i\cap\widehat{\calA}_i=2$) is a 
\demph{mixed facet} and its projection to $\RR^n$ is a \demph{mixed cell}.
Mixed facets and mixed cells are parallelepipeds (Minkowski sums of independent line segments).
\hfill$\diamond$
 \end{definition}

Huber and Sturmfels show that almost all real lifting functions are generic and the density of rational numbers
implies that there exist generic integral lifting functions.
Setting $P_i\vcentcolon=\conv(\calA_i)$ for $i=1,\dotsc,n$, then the projection to $\RR^n$ of the lower faces of $\widehat{P}$ forms a polyhedral
subdivision of the Minkowski sum $P_1+\dotsb+P_n$, called a \demph{mixed decomposition}.

This leads to a new interpretation for mixed volume.

\begin{theorem}[Huber-Sturmfels]
  Suppose that $\omega_i\colon\calA_i\to\ZZ$ for $i=1,\dotsc,n$ are generic lifting functions.
  Then the mixed volume $\MV(P_1,\dotsc,P_n)$ is the sum of the volumes of the mixed cells in the
  induced polyhedral decomposition of the Minkowski sum $P_1+\dotsb+P_n$.
\end{theorem}
\begin{proof}
  These constructions---the lifts $\widehat{P}_i$, lower faces, and the mixed subdiv\-ision---scale  multilinearly with
  positive $t_1,\dotsc,t_n\in\RR$.
  For example, a lower face $Q=Q_1+\dotsb+Q_n$~\eqref{Eq:lowerSum} of $\widehat{P}_1+\dotsb+\widehat{P}_n$ corresponds to a lower face
  $t_1Q_1+\dotsb+t_nQ_n$ of $t_1\widehat{P}_1+\dotsb+t_n\widehat{P}_n$.
  Let~$\pi\colon\RR^{n+1}\to\RR^n$ be the projection.
  This shows
  \[
  \vol(t_1P_1+\cdots + t_n P_n)\ =\ \sum_{Q} \vol(\pi(t_1Q_1+\dotsb+t_nQ_n))\,,
  \]
  the sum over all lower facets $Q$.
  By Condition~\eqref{Eq:dimSum}, $n=\dim(Q_1)+\dotsb+\dim(Q_n)$, and thus 
  \[
  \vol(\pi(t_1Q_1+\dotsb+t_nQ_n))\ =\
  t_1^{\dim(Q_1)}\dotsb t_n^{\dim(Q_n)} \vol(\pi(Q))\,.
  \]
  Hence the coefficient of $t_1\dotsb t_n$ in $\vol(t_1P_1+\dotsb+t_nP_n)$ is the sum of the volumes of the mixed cells.
\end{proof}

\begin{example}\label{Ex:MixedDecomposition}
  Let us consider this construction on our running example, using the lifts from Example~\ref{Ex:phtpy}.
  Figure~\ref{Fi:mixed} shows two views of the lower hull of the Minkowski sum $\widehat{P}+\widehat{Q}$,
  along with the mixed decomposition.
  \begin{figure}[htb]
  \centering
  \begin{picture}(140,120)(-12,0)
      \put(0,0){\includegraphics[height=120pt]{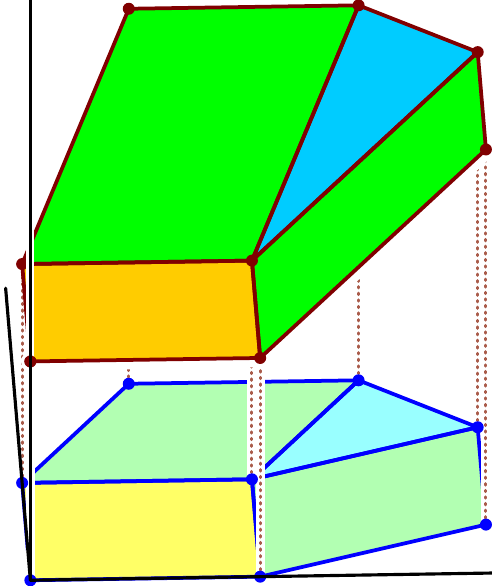}}
      \put(-4,112){\small$\omega$}   \put(-12,55){\small$x_2$}   \put(111,6){\small$x_1$}
      \put(30,30){\small$\bfalpha$}      \put(70,16){\small$\bfgamma$}
      \put(108,90){\vector(0,-1){60}}    \put(111,57){\small$\pi$}

      \put(25, 8){\small$P$}      \put(25, 52){\small$\widehat{P}$}
      \put(70,31){\small$Q$}      \put(72,101){\small$\widehat{Q}$}
  \end{picture}
  \qquad\qquad
    \begin{picture}(150,120)(-3,0)
      \put(0,0){\includegraphics[height=120pt]{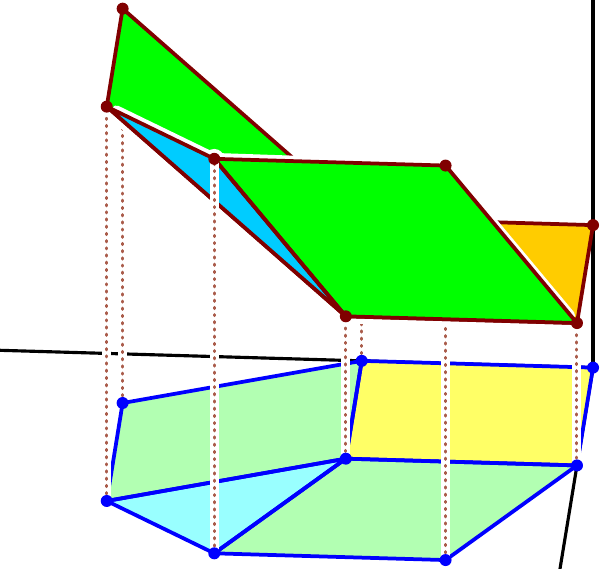}}
      \put(128,112){\small$\omega$}   \put(0,50){\small$x_1$}   \put(122,2){\small$x_2$}
      \put(78,10){\small$\bfalpha$}   \put(55,30){\small$\bfgamma$}
      \put(135,80){\vector(0,-1){40}}    \put(138,57){\small$\pi$}

      \thicklines
        \put(114.5,81){{\color{white}\vector(0,-1){16}}}
        \put(114  ,81){{\color{white}\vector(0,-1){16}}}
        \put(113.5,81){{\color{white}\vector(0,-1){16}}}
        
        \put(15,84.5){{\color{white}\vector(1,0){27}}}
        \put(15,84  ){{\color{white}\vector(1,0){27}}}
        \put(15,83.5){{\color{white}\vector(1,0){27}}}

        \put(15,6.5){{\color{white}\vector(4,1){27}}}
        \put(15,6  ){{\color{white}\vector(4,1){27}}}
        \put(15,5.5){{\color{white}\vector(4,1){27}}}
        
      \thinlines
      
      \put(101,32){\small$P$}   \put(111, 83){\small$\widehat{P}$} \put(114,81){\vector(0,-1){15}} 

      \put(4,2){\small$Q$}            \put(15,6){\vector(4,1){26}}
      \put(4,81){\small$\widehat{Q}$} \put(15,84){\vector(1,0){26}}

    \end{picture}

  \caption{Two views of the lower hull of the lift and the induced mixed subdivision with mixed cells labeled by corresponding cocharacters.}
  \label{Fi:mixed}
  \end{figure}
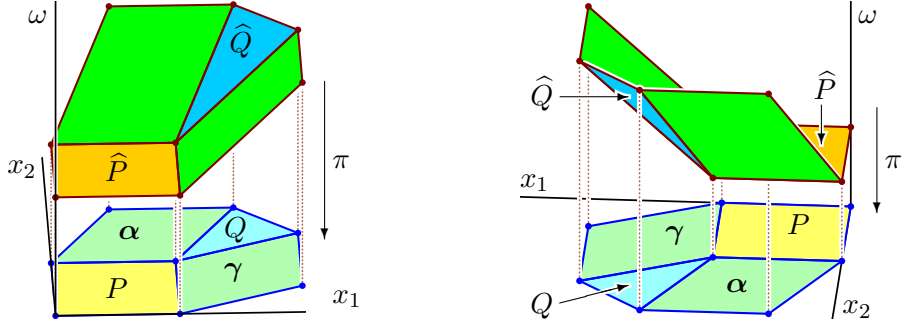
  Note that $\widehat{P}=P$ as the lifting function is 0 and $\widehat{Q}$ is affinely equivalent to $Q$.
  There are two mixed lower facets, whose corresponding mixed cells are the parallelograms of~\eqref{Eq:MinkowskiSum}, showing them to be mixed
  cells of the mixed subdivision induced by these lifts.
  The dot product with $(\bfalpha,2)=(0,-3,2)$ is minimized along the mixed lower facet
  $\conv\{(0,2,0), (2,2,0), (3,4,3), (1,4,3)\}$ with minimal value $-6$ and the dot product with $(\bfgamma,2)=(-3,0,2)$ is minimized along
  the mixed lower facet $\conv\{(2,0,0), (2,2,0), (4,3,3), (4,1,3)\}$ with minimal value $-6$.  \hfill$\diamond$
\end{example}

Keeping Example~\ref{Ex:MixedDecomposition} in mind, we return to our problem of studying the homotopy given by generic lifting functions
$\omega_i\colon\calA_i\to\ZZ$ for $i=1,\dotsc,n$, 
for the supports of our target system~\eqref{Eq:target}.
Let  $\widehat{P}\vcentcolon=\widehat{P}_1+\dotsb+\widehat{P}_n$ be the Minkowski sum of the convex hulls $\widehat{P}_i$ of the lifted
supports $\widehat{\calA}_i$.
A vector $(\bfalpha,r)\in\ZZ^{n+1}$ with $r>0$ is \demph{upward-pointing}, and the linear function $\langle(\bfalpha,r),-\rangle$ it defines
achieves its minimum on $\widehat{P}$ along a lower face $Q$---the lower face of $\widehat{P}$ \demph{exposed} by $(\bfalpha,r)$.
When $Q$ has the form~\eqref{Eq:lowerSum}, then for each $i=1,\dotsc,n$, $Q_i$ is the lower face of $\widehat{P}_i$ exposed by
$(\bfalpha,r)$, and the minimum value of $\langle(\bfalpha,r),-\rangle$ along $Q_i$ is
 \begin{equation}\label{Eq:beta}
  \min \{ \langle(\bfalpha,r), (\bfa,\omega(\bfa))\rangle = \langle \bfalpha,\bfa\rangle + r\omega(\bfa)
  \mid \bfa\in\calA_i\}\ =\ \beta_i\,,
 \end{equation}
which explains the geometric significance of $(\bfalpha,r)$ and of $\beta_i$.
When $Q$ is a facet, there is a unique primitive (components have no common factor) upward-pointing integer vector  $(\bfalpha,r)$
that exposes $Q$.
In this case, $\calA_i^{(\bfalpha)}=\pi(Q_i\cap\widehat{\calA}_i)=\calC(Q_i)$ is the support of $h^{(\bfalpha)}(\bfy,0)$.

We explain the algebraic consequences of this geometric combinatorics.
Suppose that $H^{(\bfalpha)}(\bfy,s)$ is the system of polynomials $h_i^{(\bfalpha)}(\bfy,s)$ defined by~\eqref{Eq:alpha-system}.
Then $H^{(\bfalpha)}(\bfy,0)$ is given by the polynomials $h_i^{(\bfalpha)}(\bfy,0)$  of~\eqref{Eq:alpha-facial-system}.
If, for some $i$, $\#\calC(Q_i)=1$, so that $\dim Q_i=0$, then $h_i^{(\bfalpha)}(\bfy,0)$ is a monomial and therefore
$H^{(\bfalpha)}(\bfy,0)$ has no solutions in $(\CC^\times)^n$.

Suppose that  $H^{(\bfalpha)}(\bfy,0)$ has solutions in $(\CC^\times)^n$.
Necessarily, $\dim Q_i\geq 1$ for all $i$.
By~\eqref{Eq:dimSum}, $\dim Q_i=1$ and $\#\calC(Q_i)=2$ for all $i$, and thus $Q$ is a mixed facet.
Consequently, each $h_i^{(\bfalpha)}(\bfy,0)$ is a binomial and $H^{(\bfalpha)}(\bfy,0)$ is a system of independent binomials, which may be
solved by inspection. 
Thus the only start subsystems $H^{(\bfalpha)}(\bfy,0)$ with solutions are those for which $(\bfalpha,r)$ exposes a mixed facet
$Q$ of $\widehat{P}$.
The following proposition, whose proof is sketched in Section~\ref{S:binomial}, records the number of solutions to such a mixed system.

\begin{proposition}\label{P:binomialSystem}
  The number of solutions to the system of binomials $H^{(\bfalpha)}(\bfy,0)$ is the volume of the mixed cell
  $\pi(Q)=\conv(\calC(Q_1)+\dotsb+\calC(Q_n))$.
\end{proposition}

\subsection{The Polyhedral Homotopy Algorithm}
We sketch this algorithm and provide a brief argument about its correctness. 

\begin{algorithm}
\caption{The Polyhedral Homotopy Algorithm\label{polyhedral_algo}}
\SetAlgoLined

\KwIn{A  system  $F=(f_1,\dotsc,f_n)$ of $n$ polynomials in $n$ variables, where $f_i$ has support $\calA_i$ and Newton polytope $P_i$.
The system $F$ is assumed general and has $\MV(P_1,\dotsc,P_n)$ regular solutions in $(\CC^*)^n$}.

\KwOut{All complex zeros of $F$.}
Compute generic lifting functions $\omega_i\colon\calA_i\to\ZZ$ (see Definition \ref{def_generic_lifting_function}).
They define mixed cells in the Minkowski sum $P = P_1+\dotsb+P_n$; 

\For{each mixed cell $Q$ of $P$}{
Compute the pair $(\bfalpha,r)$ given as the primitive upward pointing normal of the mixed facet of $\widehat P$ that corresponds to $Q$.

Solve the start subsystem $H^{(\bfalpha)}(\bfy,0)$ and then use homotopy
continuation to track those solutions along the homotopy $H^{(\bfalpha)}(\bfy,s)$ from $s=0$ to $s=1$, giving solutions to
$H^{(\bfalpha)}(\bfy,1)$.\label{Step_6}
}

The solutions computed in Step~\ref{Step_6} to $H^{(\bfalpha)}(\bfy,1)=H(\bfx,1)=F(\bfx)$ for all mixed cells are all the
solutions to $F(\bfx)$.
\end{algorithm}

%
%
%

\begin{proof}[Sketch of Proof of Correctness] 
  The system of equations $H(\bfx,t)=0$ defines an algebraic curve $C$ in $(\CC^\times)^n\times\CC^\times_t$ whose projection onto $\CC^\times_t$
  has degree equal to~$\defcolor{\MV}\vcentcolon=\MV(P_1,\dotsc,P_n)$ with the fiber over $t=1$ having $\MV$ points.
  This curve has $\MV$ branches near $t=0$, each of which is a point $\bfz(t)$ in $\CC\{t\}^n$.
  Here, $\CC\{t\}$ is the field of Puiseux
  series, which contains the algebraic closure of the field $\CC(t)$ of rational functions in
  $t$~\cite[Sect.\ 2.5.3]{Shafarevich}.
  Elements of $\CC\{t\}$ may be represented by fractional power series of the form
  \[
  \sum_{m\geq N} b_m t^{m/r}\,,
  \]
  where $m,N,r\in\ZZ$ with $r>0$, and $b_m\in\CC$.
  Observe that both the exponents of $t$ and the denominators in those exponents are bounded below. 

  Fix a branch $\bfz(t)$ of $C$ and let $r$ be the least common denominator of all exponents of coordinates of $\bfz(t)$.
  Consider the lowest order terms of the coordinates in $\bfz(t)$, 
  \[
     ( c_1 t^{\alpha_1/r},\dotsc, c_n t^{\alpha_n/r})\,,
  \]
  where $\alpha_i\in\ZZ$ and $r\in\NN$.
  Set $\bfalpha\vcentcolon=(\alpha_1,\dotsc,\alpha_n)$.
  The substitution $t=s^r$ clears the denominators, converting $\bfz(t)$ to a vector $\bfz(s^r)$ of Laurent series in $s$.
  The coordinate change $\bfz(s^r)\circ s^{-\bfalpha}$ converts these Laurent series to ordinary power series with constant coefficients
  $\bfc\vcentcolon=(c_1,\dotsc,c_n)$.
  Finally, $\bfc$ is a solution to the start subsystem $H^{(\bfalpha)}(\bfy,0)$.

  The point is that for each branch $\bfz(t)$ of $C$ near $t=0$, there is a weight $\bfalpha$ and positive integer $r$
  such that the vector $\bfc$ of lowest order coefficients of $\bfz(t)$ is a solution to $H^{(\bfalpha)}(\bfy,0)$.
  The discussion preceding the statement of the Polyhedral Homotopy Algorithm shows that $(\bfalpha,r)$ exposes a mixed lower facet $Q$ of
  $\widehat{P}$, and that $H^{(\bfalpha)}(\bfy,0)$ has $\vol(\pi(Q))$ solutions.
  Furthermore, each solution $\bfc$ to $H^{(\bfalpha)}(\bfy,0)$ may be developed into a power series solution $\bfy(s)$ to
  $H^{(\bfalpha)}(\bfy,s)$.
  Reversing the coordinate changes and reparameterization, this gives a solution $\bfz(t)$ to $H(\bfx,t)=0$ in $(\CC\{t\})^n$ and thus a
  branch  of the curve $C$ near $t=0$.

  Thus the homotopy paths for $H(\bfx,t)$ 
  correspond to the $\MV$ distinct branches $\bfz(t)$ of $C$ near $t=0$
  and the solutions computed in (3) give all $\MV$ solutions solutions to $F(\bfx)$.
\end{proof}

\begin{remark}\label{R:generalPHA}
  The assumption that $F(\bfx)$ is general in the Polyhedral Homotopy Algorithm ensures that $F(\bfx)$ has $\MV$ regular solutions
  {\sl and} that $H(\bfx,t)|_{t\in(0,1]}$ consists of $\MV$ smooth paths.
  Thus, to solve a given system $F(\bfx)=(f_1,\dotsc,f_n)$ where $f_i$ has support $\calA_i$, one first
  generates a general system $G=(g_1,\dotsc,g_n)$  where $g_i$ has support $\calA_i$.
  In practice, this is done by choosing random complex numbers as coefficients, and then with probability one, $G(\bfx)$ is general and
  satisfies the hypotheses of the Polyhedral Homotopy Algorithm.
  The Polyhedral Homotopy Algorithm is used to solve $G(\bfx)=0$, and then a parameter homotopy with start system $G$ and target system $F$
  is used to compute the solutions to $F(\bfx)=0$.\hfill$\diamond$
\end{remark}

\subsection{Solving binomial systems}\label{S:binomial}
To complete the discussion, we take a brief look at Step 6 in Algorithm \ref{polyhedral_algo}.
By construction, the subsystems $H^{(\bfalpha)}(\bfy,0)$ in Step 6 are binomial systems.
We explain how to solve such a system.

Suppose that $H(\bfy)$ is a system of binomials
\[
   H(\bfy) = \begin{bmatrix}
  p_1\bfy^{\bfu^{(1)}}-q_1\bfy^{\bfv^{(1)}}\ &\
    \dotsb\ &\
    p_n\bfy^{\bfu^{(n)}}-q_n\bfy^{\bfv^{(n)}}\ \end{bmatrix}
\]
where each $p_i,q_i\neq 0$ and $\bfu^{(1)}-\bfv^{(1)},\dotsc,\bfu^{(n)}-\bfv^{(n)}$ are linearly independent.
This is equivalent to the assertion that the Minkowski sum of the supports of the binomials is a parallelepiped $\pi(Q)$ of dimension $n$.
Then for $\bfy\in(\CC^\times)^n$, $H(\bfy)=0$ becomes
 \begin{equation}\label{Eq:binomial}
   \bfy^{\bfu^{(i)}-\bfv^{(i)}}\ =\ q_i/p_i\qquad\mbox{ for }i=1,\dotsc,n\,.
 \end{equation}

Let \defcolor{$A$} be the $n\times n$ matrix with rows $\bfu^{(1)}{-}\bfv^{(1)},\dotsc,\bfu^{(n)}{-}\bfv^{(n)}$.
Then $\det A = \vol(\pi(Q))$.
The \demph{Smith normal form} of $A$ consists of unimodular integer matrices~$X,Y$ (integer matrices with determinant $\pm 1$) and a diagonal matrix
$D=\mbox{diag}(d_1,\dotsc,d_n)$ such that $XAY=D$ and thus $|\det A|=|\det D| = d_1\cdot d_2\dotsb d_n$.
The unimodular matrices $X$ and $Y$ give coordinate changes on $(\CC^\times)^n$ which convert the system~\eqref{Eq:binomial} into a diagonal
system of the form
\[
   x_i^{d_i}\ =\ b_i\qquad\mbox{ for }i=1,\dotsc,n\,.
\]
All $d_1\dotsb d_n$ solutions may be found by inspection, and then the coordinate changes may be reversed to obtain all solutions to the
original system $H(\bfy)$.

%
\section{Numerical algebraic geometry}\label{sec:nag}

We have described methods to compute all isolated solutions to a system of polynomial equations.
\demph{Numerical algebraic geometry} uses this ability to compute zero-dimensional algebraic varieties to represent and manipulate
higher-dimensional algebraic varieties on a computer.
This is an essential component of numerical nonlinear algebra.
Besides expanding the reach of numerical methods, the geometric ideas behind numerical algebraic geometry have led to new
methods for solving systems of polynomial equations, including regeneration and monodromy.
While the term was coined in~\cite{Sommese96numericalalgebraic}, the fundamental ideas were developed in a series of papers
including~\cite{SOMMESE2000572,OriginalTraceTest}, and a more thorough treatment is
in~\cite[Part III]{Sommese:Wampler:2005}. 

\begin{example}\label{Ex:Reducible}
Consider the following square system of polynomials in the variables $x,y,z$:
 \begin{equation}\label{ex:nag}
   F(x,y,z)\ =\ \begin{bmatrix} f(x,y,z)g(x,y,z)(x-4)(x-6)\\ f(x,y,z)g(x,y,z)(y-3)(y-5)\\ f(x,y,z)(z-2)(z-5)  \end{bmatrix},
 \end{equation}
where
\[
f(x,y,z)\ =\ \tfrac{1}{40}(2xy-x^2) - z -1\qquad \text{and}\qquad
g(x,y,z)\ =\ x^4 - 4x^2 - y - 1\,.
\]
Figure~\ref{F:NID} shows the real part of the variety $V$ of $F(x,y,z)=0$, consisting of a quadric (degree~2) surface, two quartic (degree 4) curves
(at $z=2$ and $z=5$, respectively), and eight points.
The surface is in blue, the two curves in red, and the eight points in green.\hfill$\diamond$
\begin{figure}[htb]
  \centering
  \begin{picture}(237,165)(-10,-13)
    \put(0,0){\includegraphics[width = 200pt]{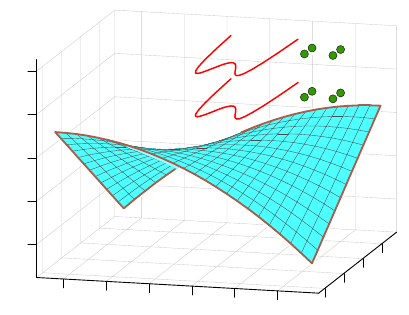}}

    \put(  5,116){\small$6$}
    \put(  5, 94){\small$3$}
    \put(  5, 72){\small$0$}
    \put( -3, 50){\small$-3$}
    \put( -3, 28){\small$-6$}
    \put(-10, 62){\small$z$}
    
    \put(16,   1){\small$-9$}    \put( 37.5,  -0.4){\small$-6$}    \put( 59,  -1.8){\small$-3$}
    \put(90,-3.2){\small$0$} \put(111.5,-4.6){\small$3$} \put(133,-6){\small$6$}
    \put(75,-13){\small$x$}

    \put(149, -4){\small$-10$} \put(164, 7){\small$-5$} \put(181, 17){\small$0$} \put(190, 24){\small$5$}
    \put(190,3){\small$y$}
 \end{picture}
  \caption{A reducible variety, defined implicitly by (\ref{ex:nag}).}
  \label{F:NID}
\end{figure}
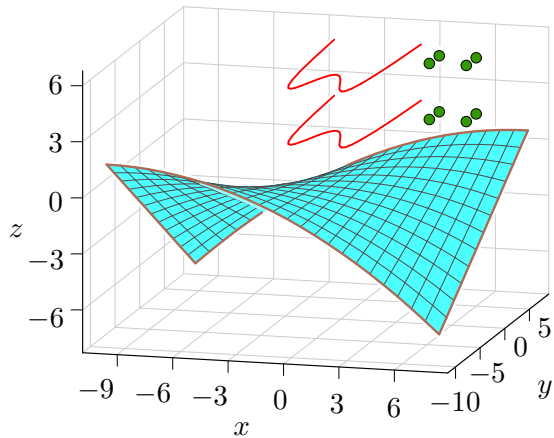
\end{example}

Given any system $F(\bfx)$ defining a reducible variety $V$, implemented symbolic algorithms (primary decomposition and computing
radicals) will decompose the variety $V$ as follows.
These methods will compute a list $I_1,\dotsc,I_r$, where each $I_i$ is the ideal of an irreducible component $V_i$
of $V$.
Each ideal $I_i$ is represented by a Gr\"obner basis, which is a finite set of generators, and thus serves as a data structure encoding information about
$V_i$. 
For example, the dimension and degree  of a component $V_i$ may be computed from the data $I_i$.

In numerical algebraic geometry, the data structure to represent a positive-dimensional component of a variety is a witness
set, which we now describe.
Suppose that \defcolor{$F(\bfx)$} is a system of polynomials, and let \defcolor{$V$} be an irreducible component of the variety defined by
$F(\bfx)=0$. 
A \demph{witness~set}~\defcolor{$W$} for the component~$V$ is a triple $W=(F, L, L \cap V)$, where $L$  is a general (affine) linear
subspace complimentary to 
$V$ in that $\codim(L)=\dim(V)$ and $L \cap V$ consists of numerical approximations of the points in the intersection of $L$ and $V$.
Generality (see Section \ref{ssec:witness}) ensures that the \demph{linear slice} $L\cap V$ is transverse and consists of $\deg(V)$ points.
In practice, $L$ is represented by $\dim(V)$ randomly chosen polynomials of degree one.

The simple algorithm of \demph{membership testing} illustrates the utility of this data structure.
Given a witness set $W=(F, L, L \cap V)$ for an irreducible variety $V\subset\CC^n$ and a point $\bfx_0\in\CC^n$, we would
like to determine if $\bfx_0\in V$.
Evaluating $F$ at $\bfx_0$ only implies that $\bfx_0$ lies near the variety defined by $F$, not that it lies near the irreducible component $V$.
We instead choose a general linear subspace $L'$ with the same codimension as $L$, but for which $\bfx_0\in L'$ (that is, $L'$ is otherwise
general, given that $\bfx_0\in L'$).
Next, form the \demph{linear slice homotopy},
 \begin{equation}\label{Eq:moveWS}
   \defcolor{H(\bfx,t)}\ \vcentcolon=\  (F(\bfx),\, tL(\bfx) + (1-t) L'(\bfx))\,,
 \end{equation}
and use it to track the points of $L \cap V$ from $t=1$ to
$t=0$, obtaining the points of $L'\cap V$.
As the intersection of $V$ with the other components of the variety of $F$ has lower dimension than $V$,
the complement in $V$ of the other components
is path-connected, and thus $\bfx_0$ lies in $L'\cap V$ if and only if $\bfx_0\in V$.  

The core of this membership test reveals another algorithm involving witness sets.
Given a witness set $W=(F, L, L \cap V)$  and a general linear subspace $L'$  with the same codimension as $L$, 
the step of following the points of $L\cap V$ along the homotopy~\eqref{Eq:moveWS} to obtain the points $L'\cap V$ is called
\demph{moving a witness set}.
This is because $W'=(F, L', L'\cap V)$ is a new witness set for $V$.
This may also be considered to be an algorithm for sampling points of $V$.

The rest of this section discusses algorithms for computing a witness set and the corresponding numerical irreducible decomposition of a
variety $V$.
It concludes with a summary of regeneration and monodromy, two new methods for solving systems of polynomials.

\begin{remark}
  The set of points in the linear slice $L\cap V$ is considered a concrete version of Andr\'e Weil's generic
  points of a variety~\cite{Weil_FAG}. 
  We call it \demph{witness point set}.

  A witness point set is related to Chow groups from intersection theory~\cite{Fulton}.
  Indeed, a witness set for an irreducible variety $V$ may be interpreted as a specific way to represent the class of $V$ in the Chow
  ring of $\PP^n$.
  In \cite{Sottile_Witness_Sets_2020} this point of view was used to extend witness sets to represent subvarieties of varieties other than
  $\PP^n$. 
\end{remark}

\subsection{More on linear slices}\label{ssec:witness}

An irreducible algebraic subvariety $V$ of affine or projective space has two fundamental invariants---its \demph{dimension},
\defcolor{$\dim(V)$}, and its \demph{degree}, \defcolor{$\deg(V)$}.
The dimension of $V$ is the dimension of its (dense subset of) smooth points, as a complex manifold.
Equivalently, this is the dimension of its tangent space at any smooth point.

By Bertini's theorem~\cite[Thm.\ 2,\ \S6.2]{Shafarevich}, there is a dense Zariski-open subset of (affine) linear spaces $L$
of codimension $\dim(V)$ such that the linear slice $L\cap V$ is transverse.
Here, a codimension $d$ linear subspace is defined by $d$ independent degree one polynomials.
The degree of $V$ is the maximal number of points in such an intersection.
By Bertini's Theorem again, this maximum is achieved by linear spaces chosen from a Zariski open subset of
the Grassmannian of linear subspaces of codimension $\dim(V)$.

In practice, $L$ is represented by $\dim(V)$ 
random degree one polynomials (their coefficients are chosen randomly).
By the nature of Zariski open sets, for most reasonable probability distributions on these coefficients, a suitably
general $L$ will be found with probability one.

When the variety $V$ defined by the vanishing of $F(\bfx)$ is reducible and the maximum dimension of an irreducible component is $d$, then a
randomly-chosen linear subspace $L$ of codimension $d$ will meet each irreducible component $V'$ of $V$ of dimension $d$ in $\deg(V')$
points~$L\cap V'$, and it will not intersect any components of $V$ of dimension less than $d$.
If $V'$ is the unique component of dimension $d$, then $(F, L, L\cap V')$ is a witness set for $V'$.

\begin{example}\label{example_witness}
We continue Example~\ref{Ex:Reducible}.
To compute the linear slice $L\cap V$ with the line $L$ parameterized by $(t,-t-2,-3+t/4)$, 
we add to $F$ two degree one polynomials $x+y+2$ and $z+3-x/4$. The augmented system defines the intersection $L\cap V$. 
It has two solutions  $(10/3,  -16/3, -13/6)$ and $(-8, 6, -5)$.
%
%
The line $L$ is sufficiently general so that it only meets the
two-dimensional surface defined by $f(x,y,z)$, and neither of the 
curves nor any isolated points.
Figure~\ref{F:Surface_section} shows this configuration. \hfill$\diamond$
  \begin{figure}[htb]
  \centering  \begin{picture}(237,165)(-10,-15)
    \put(0,0){\includegraphics[width = 200pt]{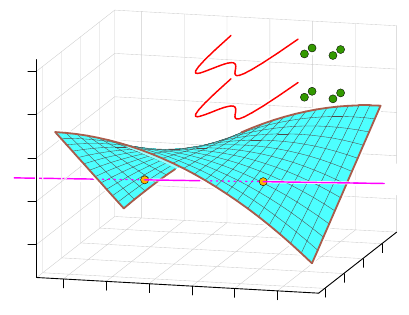}}

    \put(  5,116){\small$6$}
    \put(  5, 94){\small$3$}
    \put(  5, 72){\small$0$}
    \put( -3, 50){\small$-3$}
    \put( -3, 28){\small$-6$}
    \put(-10, 62){\small$z$}
    
    \put(16,   1){\small$-9$}    \put( 37.5,  -0.4){\small$-6$}    \put( 59,  -1.8){\small$-3$}
    \put(90,-3.2){\small$0$} \put(111.5,-4.6){\small$3$} \put(133,-6){\small$6$}
    \put(75,-13){\small$x$}

    \put(149, -4){\small$-10$} \put(164, 7){\small$-5$} \put(181, 17){\small$0$} \put(190, 24){\small$5$}
    \put(190,3){\small$y$}

    \put(190,62){\small$L$}
 \end{picture}
  \caption{Slice of $V$ by a line $L$.}
  \label{F:Surface_section}
  \end{figure}
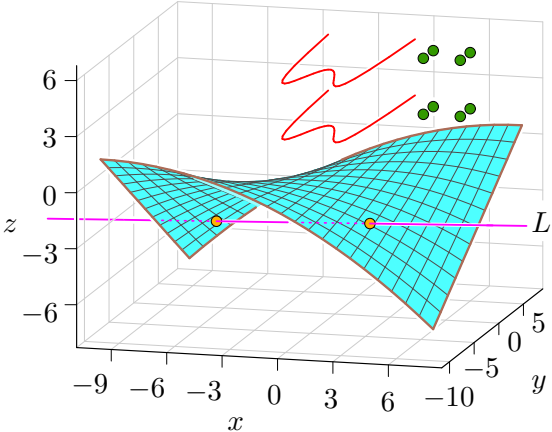
\end{example}

While finding a witness point set for the top-dimensional component of $V$ in Example~\ref{example_witness} was straightforward, finding
witness point sets for the other components is not as simple.
To find points on curves, we intersect $V$ with the vertical plane \defcolor{$P$}
defined by $x+y-2=0$,  finding eight isolated solutions.
These come from the two curves of degree four, each contributing four points.
This number eight does not yet tell us that there are two curves, there may be a single curve of degree eight or some other
configuration.
Furthermore, the plane intersects the surface in a curve \defcolor{$C$}, and we may have found additional non-isolated points
on~$C$.
This is displayed in Figure \ref{F:Surface_section_2}.
  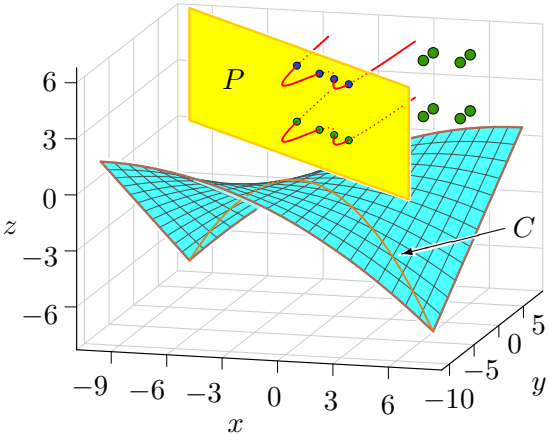
\begin{figure}[htb]
  \centering \begin{picture}(237,165)(-10,-13)
    \put(0,0){\includegraphics[width = 200pt]{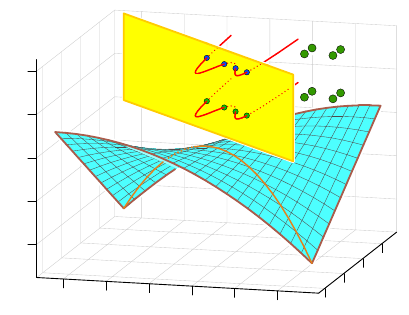}}

    \put(  5,116){\small$6$}
    \put(  5, 94){\small$3$}
    \put(  5, 72){\small$0$}
    \put( -3, 50){\small$-3$}
    \put( -3, 28){\small$-6$}
    \put(-10, 62){\small$z$}
    
    \put(16,   1){\small$-9$}    \put( 37.5,  -0.4){\small$-6$}    \put( 59,  -1.8){\small$-3$}
    \put(90,-3.2){\small$0$} \put(111.5,-4.6){\small$3$} \put(133,-6){\small$6$}
    \put(75,-13){\small$x$}

    \put(149, -4){\small$-10$} \put(164, 7){\small$-5$} \put(181, 17){\small$0$} \put(190, 24){\small$5$}
    \put(190,3){\small$y$}

    \put(73,117){\small$P$}

     \thicklines
     \put(180,64){{\color{white}\vector(-4,-1){40}}}
     \put(180,64.5){{\color{white}\line(-4,-1){38}}}
     \put(180,63.5){{\color{white}\line(-4,-1){38}}}

    \thinlines
    \put(183,61){\small$C$}\put(180,64){\vector(-4,-1){39}}
 \end{picture} \caption{The plane $P$ intersects $V$ in eight isolated points and a curve $C$.}
  \label{F:Surface_section_2}
  \end{figure}
Methods to remove points on higher-dimensional components and to determine which points lie on which components of the same
dimension are described in the next section.

%
\subsection{Numerical irreducible decomposition}\label{ssec:NID}

A system $F(\bfx)$ of polynomials  in $n$ variables defines the algebraic variety
$\defcolor{V}\vcentcolon=\{\bfx\in\mathbb C^n\mid F(\bfx)=0\}$.
Were $V$ irreducible, a witness set would be an acceptable representation for $V$.
An analog when $V$ is reducible is a \demph{numerical irreducible decomposition} of $V$.
This data structure for representing $V$ consists of a collection of witness sets $(F, L', L'\cap V')$, one for each
irreducible component $V'$ of $V$. 
We present a numerical irreducible decomposition for our running example:
\begin{gather*}
  \left(F, [x{+}y{+}2, z{+}3{-}\tfrac{x}{4}],  \{(\tfrac{10}{3}, -\tfrac{16}{3}, -\tfrac{13}{6}), (-8, 6, -5)\}\right)
  \\
  \left(F, [x{+}y{+}2],  \{
 (-2.06 , 0.06,2), (-0.40 , -1.60 , 2), (0.69 , -2.69 , 2), (1.76 , -3.76 , 2)\}\right)\,,
  \\
  \left(F, [x{+}y{+}2],  \{
  (-2.06 , 0.06,5), (-0.40 , -1.60 , 5), (0.69 , -2.69 , 5), (1.76 , -3.76 , 5)\}\right)\,,
  \\
  \left(F,[],\{(4,3,2)\}\right)\,,\  \left(F,[],\{(4,3,5)\}\right)\,,\  \left(F,[],\{(4,5,2)\}\right)\,,\  \left(F,[],\{(4,5,5)\}\right)\,,
 \\
  \left(F,[],\{(6,3,2)\}\right)\,,\  \left(F,[],\{(6,3,5)\}\right)\,,\  \left(F,[],\{(6,5,2)\}\right)\,,\  \left(F,[],\{(6,5,5)\}\right)\,.
\end{gather*}
We later present the \demph{Cascade Algorithm} (Algorithm~\ref{NID_algo}) to compute a numerical irreducible decomposition. 
We first explain its constituents.

\subsubsection{Witness point supersets}\label{SSS:WPS}
A starting point is to compute, for each $i$, a set of points $U_i$ in a linear slice $L\cap V$ of $V$ with a codimension $i$ linear space $L$,
where $U_i$ contains all witness point sets $L\cap V'$ for $V'$ an irreducible component of dimension $i$.
For this, let $\ell_1,\dotsc,\ell_{n-1}$ be randomly chosen (and hence independent) degree one polynomials.
For each $i$, let \defcolor{$L^i$} be defined by $\ell_1,\dotsc,\ell_i$, and let \defcolor{$F_i$} be a subsystem of $F$ consisting of $n-i$
randomly chosen linear combinations of elements of $F$.
Then  $(F_i,\ell_1,\dotsc,\ell_i)$ is a square subsystem of $(F,L^i)$, and we may use it to compute points \defcolor{$U_i$}
that lie in $L^i\cap V$, as explained in Section~\ref{SS:SquaringUp}.
In the literature, these sets $U_i$ are called witness point supersets.

By the generality of $\ell_1,\dotsc,\ell_{n-1}$, there will be no solutions to $(F_i,\ell_1,\dotsc,\ell_i)$ when $i$ exceeds the dimension
$d$ of $V$---yet another application of Bertini's theorem.
By the same generality, the set $U_i$ contains witness point sets for each irreducible component of $V$ of dimension $i$, and perhaps 
some points on irreducible components of $V$ of larger dimension.
The next two sections describe how to remove points of $U_i$ that lie on components of $V$ of dimension exceeding $i$, and then how to
decompose such an equidimensional slice into witness point sets for the irreducible components of $V$ of dimension $i$.

\subsubsection{Removing points on higher-dimensional components}\label{NID_step1}
Suppose that we have computed witness point supersets
$U_0,U_1,\dotsc,U_d$, where \defcolor{$d$} is the dimension of $V$ as in Section~\ref{SSS:WPS}. 
By the generality of $\ell_1,\dotsc,\ell_d$, $U_d$ is equal to the linear slice $L^d\cap V$, and thus is the union
of witness point sets for the irreducible components of $V$ of dimension $d$.
For each $i=0,\dotsc,d$, let \defcolor{$V_i$} be the union of all irreducible components of $V$ of dimension $i$.
Then $\defcolor{W_i}\vcentcolon= L^i\cap V_i\subset U_i$  consists of points in $U_i$ lying on some component of $V$ of
dimension $i$.
This union of the witness point sets of $i$-dimensional components of $V$ is an \demph{equidimensional slice}.
Also, note that $W_d=U_d$.

Since points of $U_i\smallsetminus W_i$ lie on $V_{i+1},\dotsc,V_d$, downward induction on $i$ and the membership test computes 
$U_i\smallsetminus W_i$ and thus $W_i$.
This uses the observation that the membership test, starting with $W_j=L^j\cap V_j$, may be used to determine if a point
$\bfx_0\in U_i$ lies on $V_j$, for any $j>i$. 
This is invoked in Step~\ref{Step_MT} in Algorithm~\ref{Alg:EDS} for computing such equidimensional slices.

\begin{algorithm}
\caption{Computing equidimensional slices\label{Alg:EDS}}
\SetAlgoLined

\KwIn{$F(\bfx)$, $\ell_1,\dotsc,\ell_d$, $U_1,\dotsc,U_d$ as above.}

\KwOut{Equidimensional slices $W_0,\dotsc,W_d$ of $V$.}
Set $W_d\vcentcolon=U_d$.

\For{ $i$ from $d-1$  down to $0$}{

  Set $W_i\vcentcolon=\{\}$.
  
  \For{each point $\bfx_0\in U_i$}{
   
    If $\bfx_0\not\in V_{i+1}\cup\dotsb\cup V_d$, then $W_i\vcentcolon= W_i\cup\{\bfx_0\}$.\label{Step_MT}
  }
  }

  Return $W_d,\dotsc,W_1,W_0$.
\end{algorithm}

\begin{remark}\label{Rem:alternative}
 An alternative to Algorithm~\ref{Alg:EDS} is a \demph{local dimension test}~\cite{BHPS_local_dim}, 
 which can
 determine if a point $\bfx_0\in U_i$ lies on a component of dimension exceeding $i$.\hfill
\end{remark}

\subsubsection{Decomposing equidimensional slices}
\label{NID_step2}

Suppose that we have the equidimensional slices $W_0,\dotsc,W_d$ of $V$, where $W_i=L^i\cap V_i$ for each $i$, as computed in
Algorithm~\ref{Alg:EDS}.
Fix $i$ and suppose that the irreducible decomposition of $V_i$ is
\[
V_i\ =\  X_1 \cup X_2 \cup \dotsb \cup X_r\,,
\]
so that $X_1,\dotsc, X_r$ are all of the irreducible components of $V$ of dimension $i$.
Then 
 \begin{equation}\label{Eq:WitnessSetPartition}
W_i\ =\ L^i\cap V_i\ =\ \left(L^i\cap X_1\right) \sqcup \left(L^i\cap X_2\right) \sqcup \dotsb \sqcup \left(L^i\cap X_r\right)\,.
 \end{equation}
This union is disjoint by the generality of $\ell_1,\dotsc,\ell_d$, as the intersection of $X_j\cap X_k$ with $j\neq k$
has dimension less than $i$.
Call~\eqref{Eq:WitnessSetPartition} the \demph{witness set partition} of the equidimensional slice $W_i$.
Each part $L^i\cap X_j$ is a witness point set for $X_j$.
Computing a witness set partition of $W_i$ is tantamount to computing a numerical irreducible decomposition of $V_i$.

Suppose that $H(\bfx,t)\vcentcolon=(F(\bfx), tL^i(\bfx)+(1{-}t)L'(\bfx))$ is a linear slice homotopy~\eqref{Eq:moveWS}.
As with moving a witness set,  if we track a point
$\bfx\in L^i\cap X_j$ to a point $\bfx'\in L'\cap V$, then all points of the homotopy path, including its endpoint $\bfx'$, lie on
$X_j$.

Suppose that we combine  linear slice homotopies together, moving points of $W_i=L^i\cap V_i$ to $L'\cap V_i$ on to $L''\cap V_i$,
and then back to $L^i\cap V_i$.
The three convex combinations,
 \[
 tL^i(\bfx)+(1-t)L'(\bfx)\,,\ \ tL'(\bfx)+(1-t)L''(\bfx)\,,\ \mbox{ and }\   tL''(\bfx)+(1-t)L^i(\bfx)\,,
 \]
for $t\in[0,1]$ together 
form a based loop in the space of codimension $i$ affine linear subspaces.
Tracking each $\bfx_0\in W_i$ along the three homotopies gives another point $\sigma(\bfx_0)\in W_i$.
This computes a \demph{monodromy permutation} \defcolor{$\sigma$} of $W_i$.
This has the property that the partition of~$W_i$ into the cycles of $\sigma$ refines the witness set
partition~\eqref{Eq:WitnessSetPartition}.

Following additional based loops may lead to other partitions  of $W_i$ into cycles of monodromy permutations.
The common coarsening of these monodromy cycle partitions is an \demph{empirical partition} of $W_i$.
Every empirical partition is a refinement of the witness set partition.
Since the smooth locus of $X_j\smallsetminus (\bigcup_{k\neq j}X_k)$ is path-connected,
the common coarsening of all empirical partitions is the witness set partition.
Thus computing monodromy permutations will eventually give the witness set partition.
%
%
%

The problem with this approach to numerical irreducible decomposition is that only when~$V_i$ is irreducible is there a stopping
criterion. 
Namely, if we discover an empirical partition consisting of a single part, then we conclude that $V_i$ is irreducible,
and $(F,L^i,W_i)$ is a numerical irreducible decomposition of $V_i$.
All other cases lack a stopping criterion.
That is, there is no proof that a witness set partition has been found.

A common heuristic for stopping 
is the trace test~\cite{OriginalTraceTest}.
To begin, form a linear slice homotopy~\eqref{Eq:moveWS} using a linear subspace $L'$ such that $L^i\cap L'$ has codimension $i{+}1$
in projective space.
Then the convex combination $tL^i(\bfx)+(1{-}t)L'(\bfx)$ forms a \demph{pencil} of linear spaces.
The \demph{trace test} follows from the observation that while each homotopy path $\bfx(t)$ for $t\in[0,1]$ tracked from a point
$\bfx\in W_i$ is nonlinear, if we sum over all points $\bfx\in L^i\cap X_j$ in a single part of the witness set partition, then that sum or
its average is an affine-linear function of the homotopy parameter $t$.
Given a subset $S\subset W_i$, the average of the points tracked from $S$ is the \demph{trace} of $S$.
The trace is an affine-linear function if and only if $S$ is the full witness point set~\cite[Theorem 15.5.1]{Sommese:Wampler:2005}. 

\begin{example}\label{Ex:Trace}
 Consider the folium of Descartes which is defined by $f=x^3+y^3-3xy$.
 A general line $\ell$ meets the folium in three points $W$ with the triple $(f,\ell,W)$ forming a witness set for
 the folium.
 Figure~\ref{F:folium} shows these witness sets on four parallel lines, which lie in a pencil.
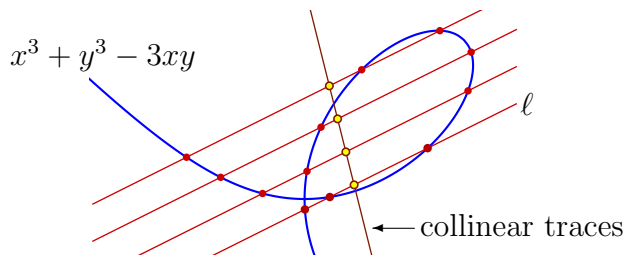
\begin{figure}[htb]
\[
   \begin{picture}(232,92)(-30,0)
     \put(0,0){\includegraphics{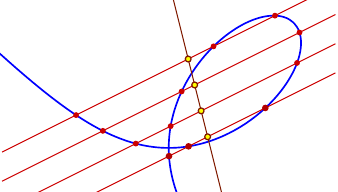}}
     \put(-30,72){$x^3+y^3-3xy$}
     \put(163,53){$\ell$}
     \put(125,7){collinear traces}
     \put(123,10){\vector(-1,0){16}}
    \end{picture}
\]
\caption{The trace test for the folium of Descartes.}
\label{F:folium}
\end{figure}
Note that the four traces are collinear.\hfill$\diamond$
\end{example}

The collinearity of traces may be seen as a consequence of Vi\`{e}ta's formula that the sum of
the roots of a monic polynomial of degree $\delta$ in $y$ is the coefficient of $-y^{\delta-1}$~\cite{traceTest}.

This gives the following heuristic for stopping 
for computing a numerical irreducible decomposition.
Given a part $S$ of an empirical partition of $W_i$, track all points of $S$ along a linear slice homotopy given by a
pencil containing $L^i$.
If the traces are collinear, then $S$ is a witness point set for some component of $V_i$.
Otherwise, either compute more monodromy permutations to coarsen the empirical partition or check the collinearity of the
trace for the union of $S$ with other  parts of the empirical partition.
This is called the \demph{trace test}.
The reason this is only a heuristic is that there is no way to certify collinearity of traces of approximate solutions.
In \cite{BrBu23} the authors develop a trace test specifically for systems of polynomials that are sparse in the sense of
Section~\ref{polyhedral-homotopy}.

\begin{example}\label{Ex:Trace_Decomposition}
Suppose that $V$ is the union of the ellipse $8(x+1)^2+ 3(2y+x+1)^2 = 8$ and the folium, as in
Figure~\ref{F:folium_ellipse}.
\begin{figure}[htb]
\[
   \begin{picture}(192,108)
     \put(0,0){\includegraphics{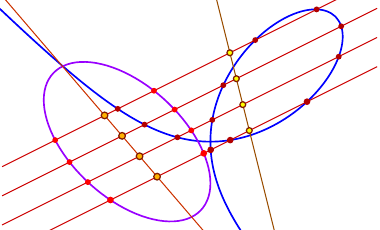}}
     \put(184,74){$\ell$}
    \end{picture}
\]
\caption{Numerical irreducible decomposition for the ellipse and folium.}
\label{F:folium_ellipse}
\end{figure}
A witness set for $V$ consists of the five points $W=V\cap\ell$.
Tracking points of $W$ as $\ell$ varies over several loops in the space of lines gives an empirical
partition of $W$ into two sets, of cardinalities two and three, respectively.
Applying the trace test to each subset verifies that each is a witness set of a component of $V$.\hfill$\diamond$
\end{example}


The methods described in this and the previous sections combine to give the \demph{Cascade Algorithm} for computing a numerical
irreducible decomposition.
This was introduced in~\cite{SOMMESE2000572} and is implemented in \texttt{PHCpack}~\cite{PHCpack}, \texttt{Bertini}~\cite{Bertini} and
\texttt{HomotopyContinuation.jl}~\cite{HC.jl}, and in \texttt{NAG4M2}~\cite{NumericalAlgebraicGeometryArticle} through its interfaces.
We present a simplified form in Algorithm~\ref{NID_algo}.

\begin{algorithm}
\caption{The Cascade Algorithm\label{NID_algo}}
\SetAlgoLined

\KwIn{A system $F(\bfx)$ of polynomials in $n$ variables defining a variety $V$ of dimension $d$.}

\KwOut{A numerical irreducible decomposition of $V$.}

\For{each dimension $i$ from $d$ down to $0$}{

Choose a codimension $i$ linear space $L^i$ and compute $L^i\cap V$, yielding points $U_i$.

Remove from $U_i$ any points lying on higher-dimensional components as in Algorithm~\ref{Alg:EDS}. Call the remaining points $W_i$.

Compute the witness set partition $W_i=S_1\sqcup\cdots\sqcup S_r$ using monodromy and the trace test as explained in~\Cref{NID_step2}.

Return $(F,L^i,S_j)$ for $j=1,\ldots,r$.  These are witness sets for the irreducible components of $V$ of dimension $i$.
}
\end{algorithm}

\subsection{Advanced methods for solving}\label{ssec:advanced}

The perspective afforded by numerical algebraic geometry and its tools---witness sets and monodromy---lead to new algorithms for solving
systems of equations.
We describe two such algorithms.
Regeneration~\cite{Regen} is a bootstrap method that constructs a numerical irreducible decomposition one equation at a time.
Monodromy solving~\cite{Monodromy_Solve} exploits that it is often easier to find a system of equations for which a given point is a
solution than to find a solution to a given system.

\subsubsection{Regeneration}\label{ssec:regen}

Let $F\vcentcolon=(f_1,\dotsc,f_m)$ be a system of polynomials in $n$ variables. 
Rather than solve all equations at once as in~\eqref{eq:AlgVariety}, we instead consider
the sequence of varieties
\[
\CC^n\ =\  V_0 \supset V_1 \supset V_2 \supset \dotsb \supset V_m = V\,,
\]
where $\defcolor{V_i}$
is defined by the first $i$ polynomials in $F$.
The approach of \demph{equation-by-equation solvers}~\cite{EqByEq,Regen,RegenCascade,DLR24} is to iteratively compute
$V_i$ given $V_{i-1}$ for each $i = 1,\dotsc,m$.

Let $X\subset V_{i-1}$ be an irreducible component of dimension $d$.
Then
\[
X\cap V_i\ =\ \{\bfx\in X\mid f_i(\bfx)=0\}\,.
\]
If $f_i$ vanishes identically on $X$, then $X$ is an irreducible component of $V_i$.
Otherwise $X\cap V_i$, if nonempty, is a union of irreducible components of $V_i$ of dimension $d{-}1$.
We explain how to obtain a numerical irreducible decomposition of $X\cap V_i$ given a witness set for $X$.

Let $((f_1,\dotsc, f_{i-1}), L^d, L^d\cap X)$ be a witness set for $X$.
By the generality of $L^d$, with probability one we may conclude that a general polynomial $f$ vanishes on $X$ only if $f$ vanishes 
at each point of $W\vcentcolon=L^d\cap X$.
\demph{Regeneration} is a method to use $W$ to compute a witness point superset for $X\cap V_i$.
Let $\ell_1,\dots,\ell_d$ be $d$ linear polynomials defining  $L^d$, and suppose that \defcolor{$\delta$} is the degree of $f_i$.
Form the convex combination $\defcolor{\ell(t)}\vcentcolon= t \ell_d + (1-t)\ell'$ of $\ell_d$ with a new general degree one polynomial,
\defcolor{$\ell'$}.
Use the straight-line homotopy
\[
(f_1,\dotsc, f_{i-1}\,,\ \ell(t)\,,\ \ell_1,\dotsc,\ell_{d-1})
\]
to move the witness point set $W_1=W=L^d\cap X$ at $t=1$ to witness point sets $W_2,\dotsc,W_\delta$ at distinct points
$t=t_2,\dotsc,t_\delta$, respectively.

Then the product $f\vcentcolon= \ell_d\cdot \ell(t_2)\dotsb \ell(t_\delta)$ has degree $\delta$ in $\bfx$ and we have
\[
\defcolor{U'}\ =\ W_1\cup W_2 \cup\dotsb\cup W_\delta
\ =\ L^{d-1}\cap \left(X  \cap\{\bfx \mid f(\bfx)=0\}\right)\,,
\]
where $L^{d-1}$ is defined by $\ell_1,\dotsc,\ell_{d-1}$. 
Use the straight-line homotopy
\[
(f_1,\dotsc,f_{i-1}\,,\ t f + (1-t) f_i\,,\ \ell_1,\dotsc,\ell_{d-1})
\]
to track the points of $U'$ at $t=1$ to the set $U$ at $t=0$.
Then
\[
   U\ =\ L^{d-1}\ \cap\  \bigl(X\ \cap\ V_i\bigr)
\]
is a witness point superset for $X\cap V_i$.
Finally, use monodromy and the trace test of Section~\ref{NID_step2} to decompose $U$ into witness point sets for the irreducible components of $X\cap V_i$.

As regeneration computes a numerical irreducible decomposition of the variety $V$ of $F$, it will also compute all isolated  solutions to
$F$.

\subsubsection{Solving by monodromy}\label{SS:Monodromy}

Suppose that we wish to solve a system $F=F(\bfx;\bfp)$ of polynomials that lies in a parameter space of polynomial systems as in
Section~\ref{SS:parameterContinuation}, and that evaluation at a general point $\bfx_0\in\CC^n$ gives $n$ independent linear equations in
the parameters $\bfp\in\CC^k$.
For example, $F(\bfx;\bfp)$ could be the family of all polynomial systems $f_1(\bfx),\dotsc,f_n(\bfx)$ where the degree of $f_i$ is $d_i$
and $\bfp\in\CC^k$ is the vector of coefficients.
More generally, each $f_i(\bfx)$  could be a sparse polynomial with support $\calA_i$.

Consider the incidence variety~\eqref{Eq:Incidence_I} with projections $\pi_1$ to $\CC^n$ and $\pi_2$ to $\CC^k$.
 \begin{equation}\label{Eq:Incidence_II}
   \raisebox{-27pt}{\begin{picture}(290,61)(3,0)
       \put(20,49){$Z\ \vcentcolon=\ \{ (\bfx, \bfp) \in \CC^n \times \CC^k \mid F(\bfx;\bfp) = 0\}\ \subseteq\ \CC^n \times \CC^k$}
       \put(21,45){\vector(-1,-4){8}} \put( 3,30){\small$\pi_1$} \put( 5,0){$\CC^n$}
       \put(26,45){\vector( 1,-4){8}} \put(34,30){\small$\pi_2$} \put(31,0){$\CC^k$}
   \end{picture}}
 \end{equation}
For any parameters $\bfp\in\CC^k$, $\pi_2^{-1}(\bfp)$ is the set of solutions $\bfx\in\CC^n$ to $F(\bfx;\bfp)=0$.
On the other hand, if we fix a general $\bfx_0\in\CC^n$, then $\pi_1^{-1}(\bfx_0)\subset\CC^k$ is defined by $n$ independent
linear equations on $\CC^k$, and is thus a linear subspace of dimension $k{-}n$.
(This implies that $Z$ is irreducible and has dimension $k$, which explains why the general fiber $\pi_2^{-1}(\bfp)$ is finite and
$\pi_2\colon Z\to\CC^k$ is a branched cover.)
Imposing $k{-}n$ additional general degree one equations on~$\pi_1^{-1}(\bfx_0)$ gives a single parameter value $\bfp_0\in\CC^k$ such that
$F(\bfx_0;\bfp_0)=0$, that is, a system of polynomials $F(\bfx;\bfp_0)$ 
for which $\bfx_0$ is a solution.

The underlying idea of monodromy solving~\cite{Monodromy_Solve} is to use monodromy to discover all solutions to
$F(\bfx;\bfp_0)=0$ and then use a parameter homotopy to find solutions to any desired system of polynomials in the family.
Similar to the description of monodromy in Section~\ref{NID_step2}, if we choose general points $\bfp_1,\bfp_2\in\CC^k$, we may form the
trio of parameter homotopies 
$F(\bfx; t\bfp_0+(1-t)\bfp_1)$, 
$F(\bfx; t\bfp_1+(1-t)\bfp_2)$, and
$F(\bfx; t\bfp_2+(1-t)\bfp_0)$.
For $t\in[0,1]$, these form a loop in the parameter space based at $\bfp_0$, and we may track the point $\bfx_0$ along this loop to obtain a
possibly new point $\bfx'\in\pi_2^{-1}(\bfp_0)$ so that $F(\bfx';\bfp_0)=0$.

More generally, given a subset $S\subset\pi_2^{-1}(\bfp_0)$ of computed points, we may track it along a possibly new loop in $\CC^k$ based
at $\bfp_0$ to obtain a subset $S'\subset\pi_2^{-1}(\bfp_0)$.
Thus we may discover additional solutions to $F(\bfx;\bfp_0)=0$.

When the number $N$ of solutions  to a general system in the family is known (e.g.,\ via Bernstein's bound for the sparse systems of
Section~\ref{S:Bernstein}), this method has a stopping criterion.
Otherwise, some heuristic may be used once sufficiently many solutions are known.
The technique of solving using monodromy was introduced in~\cite{Monodromy_Solve}, where a more complete description may be found.
It is implemented in the package \texttt{MonodromySolver} of \texttt{Macaulay2}~\cite{M2} and \texttt{HomotopyContinuation.jl}~\cite{HC.jl}
and widely used, in particular when it is not necessary to compute all solutions to a given system.

Suppose that we have a branched cover $\pi\colon Z\to Y$ with $Y$ rational (e.g.\ as in~\eqref{Eq:Incidence_II} where $Y=\CC^k$), and we know all
solutions $\pi^{-1}(\bfy_0)$ for a parameter value $\bfy_0$ not in the branch locus, $B$.
As in Section~\ref{NID_step2}, tracking all points in $\pi^{-1}(\bfy)$  as $\bfy$ varies along a loop in $Y\smallsetminus B$ based at
$\bfy_0$ gives a monodromy permutation $\sigma$ of $\pi^{-1}(\bfy_0)$, which we regard as an element of the symmetric group $S_N$, where
$N=|\pi^{-1}(\bfy_0)|$.
The set of all monodromy permutations forms the \demph{monodromy group} of the branched cover $Z$.

This is in fact a Galois group~\cite{Harris,J1870,GGEGA,Vakil}:
Let $\KK=\CC(Y)$ be the field of rational functions on the parameter space $Y$ and let $\LL=\CC(Z)$ be the function field of the incidence
variety $Z$.
As $\pi$ is dominant, we may regard $\KK$ as a subfield of $\LL$ via $\pi^{-1}$, and $\LL/\KK$ is a field extension of degree $N$.
The Galois group of the normal closure of $\LL/\KK$ is equal to the monodromy group of
$Z$, and we call it the Galois group of the branched cover $Z$.

There are several approaches to computing Galois groups using methods from numerical nonlinear algebra.
In~\cite{Galois_LS}, monodromy permutations were computed and used to show some Galois groups were the full symmetric group (see
Section~\ref{ssec:enum}).
Other approaches, including methods guaranteed to compute generators of Galois groups, were developed in~\cite{Ngalois}.
Yahl~\cite{Yahl_Fano} introduced a method to certify that a Galois group contains a simple transposition, using ideas from this
section and from Section~\ref{sec:cert}.

A Galois group that is imprimitive (preserves a partition of the solutions) is equivalent to the branched cover decomposing as a composition
of branched covers, and this may be exploited for solving (computing points in a fiber).
This is explained in~\cite{amendola2016solving,brysiewicz2021solving}.

\section{Certification}\label{sec:cert}

Let $F$ be a square system of polynomials and let $\bfz_0$ be a point presumed to be an approximation of a solution to $F$. 
We discuss methods that can give a computational proof
that Newton iterates starting from $\bfz_0$ converge to a nearby regular zero $\bfz$ of $F$.
Such methods \demph{certify} the numerical solution $\bfz_0$ to $F$.
Certification methods can also prove that two numerical solutions correspond to two distinct zeros, and are thus a
useful tool in both theoretical and applied problems in numerical nonlinear algebra.

There are  two main strategies to certify solutions to square polynomial systems,  \demph{Smale's $\alpha$-theory} and
\demph{Krawczyk's method}.
A difference  is that Smale's $\alpha$-theory uses exact arithmetic, while Krawczyk's method uses floating-point arithmetic, and is
  typically more efficient. 

\begin{remark}
  There are other approaches.
  In \cite{DUFF2022367} the authors develop methods to certify overdetermined systems which require global information.
  In~\cite{HHS}, overdetermined systems are reformulated as square systems to enable certification.\hfill$\diamond$
\end{remark}

\subsection{Smale's \texorpdfstring{$\alpha$}{alpha}-theory}

Smale's $\alpha$-theory certifies approximate zeros of a square system~$F$.
An approximate zero of $F$ is a data structure representing a solution to $F$.
We recall its definition from Section~\ref{sec:tracking}.

\begin{definition}\label{def_approx_zero} 
  Let $F(\bfx)$ be a square system of $n$ polynomials in $n$ variables.
  Writing $\defcolor{JF}\vcentcolon=\frac{\partial F}{\partial \bfx}$ for its Jacobian matrix, its
  Newton operator $N_F$~\eqref{Newton} is 
  $N_F(\bfx)\vcentcolon=\bfx - \left(JF(\bfx)\right)^{-1} F(\bfx)$.
  A point $\bfz_0 \in \CC^n$ is an \demph{approximate zero} of $F$ if there exists
  a regular zero $\bfz\in \CC^n$ of $F$ such that the sequence $\{ \bfz_k\mid k\geq 0\}$ of Newton iterates
  defined by  $\bfz_{k+1} = N_F(\bfz_k)$ for $k\geq 0$ converges quadratically to $\bfz$ in that 
  \[
  \| \bfz_{k+1} - \bfz\|\ \leq \   \frac{1}{2} \| \bfz_k-\bfz\|^2
  \qquad \forall k\,\geq\,0\,.
  \]
  We call $\bfz$ the \demph{associated zero} of $\bfz_0$.\hfill$\diamond$
\end{definition}

Smale's $\alpha$-theory certifies that a point $\bfx$ is an approximate zero using only local information encoded in 
two functions of $F$ and $\bfx$,
 \[
   \defcolor{\beta(F,\bfx)}\ \vcentcolon=\ \| JF(\bfx)^{-1}F(\bfx)\|
   \ \ \text{ and }\ \ 
   \defcolor{\gamma(F,\bfx)}\  \vcentcolon=\  
    \max_{k\geq 2}\Big\| \frac{1}{k!}\, JF(\bfx)^{-1} \mathrm D^kF(\bfx)\Big\|^\frac{1}{k-1}\,.
 \]
Here, $\beta$ is the size of a Newton step, 
$\mathrm D^kF(\bfx)$ is the symmetric tensor of derivatives of order $k$ at $\bfx$, and $JF(\bfx)^{-1} \mathrm D^kF(\bfx)$ is the
corresponding multilinear map $(\CC^n)^k\to \CC^n$.
%
%
The norm is the operator norm $\| A\| \vcentcolon= \max_{\| v \| = 1} \| A(v,\ldots,v)\|$.

Let $\defcolor{\alpha(F,\bfx)}\vcentcolon=\beta(F,\bfx) \cdot \gamma(F,\bfx)$ be the product of these two functions.
We state two results of Smale \cite[Theorem 4 and Remark 6 in Chapter 8]{BCSS}.

\begin{theorem}\label{alpha_theorem}
 Let $\bfx\in\CC^n$ and let $F$ be a system of $n$ polynomials in $n$ variables.
 \begin{enumerate}[label=(\alph*)]
 \item If $\alpha(F,\bfx)<\frac{13-3\sqrt{17}}{4}\approx 0.15767$, then $\bfx$ is an approximate zero of $F$
        whose associated zero $\bfz$ satisfies $\|\bfx-\bfz\|\leq 2\beta(F,\bfx)$.
  \item  \label{alpha_theorem_b}
      If $\bfx$ is an approximate zero of $F$ and $\bfy\in\CC^n$ satisfies $\Vert \bfx-\bfy\Vert < \tfrac{1}{20\,\gamma(F,\bfx)}$,
    then $\bfy$ is also an approximate zero of $F$ with the same associated zero as $\bfx$.
 \end{enumerate}
\end{theorem}

Shub and Smale \cite{SS99} derived an upper bound for $\gamma(F,\bfx)$ which can be computed using exact arithmetic, and thus one
may decide algorithmically if $\bfx$ is an approximate zero of $F$, using only data of $F$ and the point $\bfx$ itself.

The software \texttt{alphaCertified}~\cite{alpha_certify} uses this theorem in an algorithm.
An implementation  is publicly available for
download\footnote{\url{https://franksottile.github.io/research/stories/alphaCertified/index.html}}. 
If the polynomial system $F$ has only real coefficients, then \texttt{alphaCertified} can decide if an associated zero is real.
The idea is as follows.
We write $\overline{\bfx}$ for the complex conjugate of a point $\bfx\in\CC^n$.

Suppose that $\bfx\in \CC^n$ be an approximate zero of $F$ with associated zero $\bfz$.
Since the Newton operator has real coefficients, $N_F(\overline{\bfx}) =\overline{ N_F(\bfx)}$, we see that $\overline{\bfx}$ is an
approximate zero of $F$ with associated zero~$\overline{\bfz}$.
Consequently, if $\| \bfx-\overline{\bfx}\| < \tfrac{1}{20\,\gamma(F,\bfx)}$, then $\bfz=\overline{\bfz}$ by
\Cref{alpha_theorem}\ref{alpha_theorem_b}. 

\subsection{Krawczyk's method} 
Interval arithmetic certifies computations using floating-point arithmetic. 
\demph{Krawczyk's method} \cite{Krawczyk:1969} adapts Newton's method to interval arithmetic and can certify zeros of
analytic function $\CC^n\to\CC^n$.
This is explained in \cite{Rump83}.

Real interval arithmetic involves the set of compact real intervals,
\[
    \defcolor{\IR}\ \vcentcolon=\  \{[x,y]\mid x,y\in\mathbb R, x\leq y\}\,.
\]
For $X, Y \in \IR$ and $\circ \in \{ +,-, \cdot, \div \}$, we define
$\defcolor{X \circ Y} \vcentcolon= \{ x \circ y \,|\, x\in X,y\in Y\}$.
(For $\div$ we require that $0\not\in Y$.)
For intervals $I,J,K\in \IR$ we have $I\cdot (J + K) \subseteq I\cdot J + I\cdot K$, but the inclusion may be strict.
Indeed,
\begin{align*}
    [0,1] \cdot ( [-1,0] + [1,1] )\ &=\ [0,1]\cdot [0,1] \ =\ [0,1]\,\ \ \mbox{ but}\\
   [0,1] \cdot [-1,0] + [0,1] \cdot [1,1] \ &=\  [-1,0] + [0,1]\ =\ [-1,1]\,.
\end{align*}
Thus there is no distributive law in interval arithmetic.

\demph{Complex intervals} are rectangles in the complex plane of the form
\[
X+\sqrt{-1}Y\  =\ \{x+\sqrt{-1}y\mid x\in X,y\in Y\}\,, \quad
 \mbox{where}\ X,Y\in\IR\,.
\]
Let \defcolor{$\IC$} be the set of all complex intervals.
Writing $\frac{X}{Y}$ for $X\div Y$, 
we define arithmetic for complex intervals $I=X+\sqrt{-1}\ Y$ and $J=W+\sqrt{-1}Z$ as follows.
 \begin{alignat*}{2}
   I + J &\vcentcolon=\ (X + W) + \sqrt{-1}\ (Y+Z)
    \qquad I \cdot J &&\vcentcolon=\ (X\cdot W - Y \cdot Z) + \sqrt{-1}\ (X\cdot Z + Y\cdot W)\\ 
    I - J &\vcentcolon=\ (X - W) + \sqrt{-1}\ (Y-Z)\qquad 
    \;\, \frac{I}{J} &&\vcentcolon=\
      \frac{X \cdot W + Y\cdot Z}{W\cdot W + Z\cdot Z} + \sqrt{-1}\ \frac{Y \cdot W - X\cdot Z}{W\cdot W + Z\cdot Z}
 \end{alignat*}
As before, for $\tfrac{I}{J}$ we assume that $0\not\in (W\cdot W + Z\cdot Z)$.

 As with real intervals, there is no distributive law for complex intervals.
 Consequently, evaluating a polynomial at intervals is not well-defined.
 Evaluation at intervals is well-defined for expressions of a polynomial as a straight-line program,
 which is an evaluation of the polynomial via a sequence of arithmetic operations that does not involve distributivity.

\begin{example}\label{example_IA}
  Consider the polynomial $f(x,y,z)=x(y+z) = xy+xz$.
  These two expressions of the distributive law are different straight-line programs for $f$, and we have shown that they have distinct
  evaluations  on the triple $( [0,1], [-1,0], [1,1])$.\hfill$\diamond$
\end{example}

We sidestep this issue with the notion of an interval enclosure.

\begin{definition}
  Let $F$ be a system of $m$ polynomials in $n$ variables. We call a map
  \[
  \defcolor{\square F}\colon (\IC)^n \rightarrow (\IC)^m
  \]
  such that
  $\{F(\bfx) \mid \bfx \in \bfI\} \subseteq \square F(\bfI)$ for every $\bfI \in (\IC)^n$ an \demph{interval enclosure} of
  $F$.\hfill$\diamond$ 
\end{definition}
Let~$\square F$ be an interval enclosure of a square polynomial system $F$ and let $\square JF$ be an interval enclosure of its
Jacobian  map $JF\colon\CC^n\to\CC^{n\times n}$.  
Furthermore, let $\bfI\in (\IC)^n$, $\bfx \in \CC^n$, and let $Y \in \CC^{n\times n}$ be an invertible matrix.
The \demph{Krawczyk operator} these define is
\[\defcolor{K_{\bfx,Y}(\bfI)}  \vcentcolon=\
  \bfx - Y \cdot \square F(\bfx) + (\mathbf{1}_n - Y \cdot \square \mathrm{J}F(I))(\bfI-\bfx)\,,
\]
where $\mathbf{1}_n$ is the $n\times n$ identity matrix.
The norm of a matrix interval $A\in(\IC)^{n\times n}$ is
$\| A\|_\infty \vcentcolon= \max\limits_{B\in A} \max\limits_{\bfv\in\CC^n} {\| B\bfv\|_\infty} / {\| \bfv\|_\infty}$,
where $\Vert (v_1,\ldots,v_n) \Vert_\infty = \max_{1\leq i\leq n}\vert v_i\vert$ for $\bfv\in\CC^n$.

We state the main theorem underlying Krawczyk's method, which is proven in~\cite{Rump83}.

\begin{theorem}\label{main_theorem_IA}
  Let $F=(f_1,\ldots,f_n)$ be a system of $n$ polynomials in $n$ variables, $\bfI\in(\IC)^n$, $\bfx\in \bfI$,
  and let $Y\in\CC^{n\times n}$ be invertible.
\begin{enumerate}[label=(\alph*)]
\item \label{IA_a} If $K_{\bfx,Y}(\bfI) \subset \bfI$, there is a zero of $F$ in $\bfI$.
\item If $\sqrt{2} \, \| \mathbf{1}_n - Y\cdot  \square JF(\bfI) \|_\infty < 1$, then $F$ has a unique zero in $\bfI$.
\end{enumerate}
\end{theorem}

Several choices have to be made to implement Krawczyk's method.
For instance, we have to choose interval enclosures of both $F$ and its Jacobian $JF$.
\Cref{example_IA} shows that this is nontrivial as different straight-line programs for the same polynomial system can produce different
results in interval arithmetic.
Furthermore, choosing $\bfI$ in \Cref{main_theorem_IA} too small might cause the true zero not to lie in $\bfI$, while choosing $\bfI$ too
large can be an obstacle for the contraction property in Theorem~\ref{main_theorem_IA}\ref{IA_a}.
Heuristics are usually implemented to address these issues. 

The key idea for combining interval arithmetic with floating point operations is that one can round \demph{outwards}.
Rouding outwards, an interval computed in floating point arithmetic contains the  exact interval.
More specifically, suppose that $X, Y \in \IC$ are intervals and we want to compute $X \circ Y$, $\circ\in\{+,-,*,/\}$.
Let $u$ denote the machine precision.
On the computer we get the interval $\mathrm{fl}(X \circ Y) \vcentcolon= \{ (x  \circ  y)(1+\delta) \mid |\delta| \le u, x\in X,y\in Y \}$,
which \demph{contains} $X\circ Y$.
Consequently, it can happen that the Krawczyk operator $K_{\bfx,Y}$
is a contraction for an interval $I$, but that machine arithmetic cannot verify this, because $\mathrm{fl}(X  \circ  Y)$ is too large.

Krawczyk's method is implemented 
in \texttt{HomotopyContinuation.jl} \cite{HC.jl,BKT2020},
\texttt{Macaulay2}~\cite{M2}  package \texttt{NumericalCertification}~\cite{LeeM2,BLL:2019},
and in the commercial \texttt{MATLAB} package \texttt{INTLAB} \cite{Rump1999}.
Krawczyk's method can also certify the reality of a zero:
Assume that $F$ has real coefficients.
Suppose that we have found an interval $\bfI \in (\IC)^n$ and a matrix $Y\in \CC^{n\times n}$ such that
$K_{\bfx,Y}(\bfI) \subset \bfI$ and $\sqrt{2} \, \| \mathbf{1}_n - Y\cdot \square JF(\bfI) \|_\infty < 1$.
By Theorem~\ref{main_theorem_IA}, $F$ has a unique zero $\bfz$ in $\bfI$.
Since $\overline{\bfz}$ is also a zero of $F$, if $\{ \overline{\bfy} \mid \bfy\in K_{\bfx,Y}(\bfI) \} \subset \bfI$, then
$\bfz = \overline{\bfz}$.

%
%
%

\section{Applications}\label{sec:apps}

While we have largely discussed the theory and many aspects, methods, and some implementations of numerical nonlinear algebra, these were all
motivated by its applications to questions within and from outside of mathematics.
Many of these are well-known and may be found in other contributions in this volume.
We present three applications here, involving synchronization of oscillators, enumerative geometry, and computer vision.

\subsection{The Kuramoto model}\label{ssec:kuramoto}
In his 1673 treatise on pendulums and horology~\cite{Huygens1673Horologium}, Christiaan Huygens observed an ``odd kind of sympathy'' between
pendulum clocks, which was one of the earliest observations of synchronization among coupled oscillators.
Other examples range from pacemaker cells in the heart to the formation of circadian rhythm in the brain
to synchronized flashing of fireflies.
The Kuramoto model emerged from this study and has found many interesting applications.

A network of oscillators can be modeled as a swarm of points circling the origin which  pull on each other.
For weakly coupled and nearly identical oscillators, the natural separation of timescales \cite{Winfree1967Biological,Kuramoto1975Self} 
allows a simple description of the long-term behavior
in terms of phases of the oscillators.
Kuramoto singled out the simplest case governed by equations
\begin{equation}\label{equ: generalized kuramoto}
  \dot{\theta}_i\ =\ 
  \omega_i -
  \sum_{j \sim i} k_{ij} \sin(\theta_i - \theta_j)
  \quad\text{ for } i = 0,\ldots,n\,.
\end{equation}
Here, $\defcolor{\theta_0},\ldots,\defcolor{\theta_n} \in [0,2\pi)$ are the phases of the oscillators,
\defcolor{$\omega_i$} are their natural frequencies, $\defcolor{k_{ij}} = \defcolor{k_{ji}}$ are coupling coefficients,
and $\defcolor{j\sim i}$ is adjacency in the  graph $G$ underlying the network.
This is the Kuramoto model \cite{Kuramoto1975Self}.
It is simple enough to be analyzed yet it exhibits interesting emergent behaviors,
and has initiated  an active research field \cite{Strogatz2000From}.

One core problem that can be studied algebraically is  \demph{frequency synchronization}.
This occurs when the tendency of oscillators to relax to their limit cycles
and the influences of their neighbors reach equilibrium,
and the oscillators are all tuned to their mean frequency.
Such synchronized configurations correspond to
equilibria of \eqref{equ: generalized kuramoto},
which are solutions to a nonlinear system of equations.
Even though this system is derived from a simplification of the oscillator model,
its utility extends far beyond this narrow setting.
For example, in electrical engineering, 
it coincides with a special case of the power flow equations,
derived from laws of alternating current circuits~\cite{DorflerBullo2014Synchronization}.

The equilibrium equations become algebraic after a change of variables.
Numerical nonlinear algebra has been used to solve these and related families of equations
finding synchronization configurations that cannot be found by simulations or symbolic computation.
For example, the IEEE 14 bus system from electrical engineering is a well-studied test case,
yet its full set of solutions remained unknown until it was computed using
total degree and polyhedral homotopy methods by Mehta, {\it et al.}~\cite{MehtaNguyenTuritsyn2016Numerical},
using \texttt{Bertini} \cite{Bertini} and \texttt{HOM4PS-2.0} \cite{LeeLiTsai2008HOM4PS}.

For rank one coupling, Coss, {\it et al.}~\cite{CossHauensteinHongMolzahn2018Locating}
showed that the equilibrium equation of \eqref{equ: generalized kuramoto}
may be reformulated as a set of decoupled univariate radical equations,
which are easy to solve.

Determining the number of complex equilibria (solutions to the equilibrium equations~\eqref{equ: generalized kuramoto})
is another line of research that has used numerical nonlinear algebra.
In the 1980s, Baillieul and Byrnes \cite{BaillieulByrnes1982Geometric}
showed that a complete network of three oscillators has at most six complex equilibria, and all may be real.
For a complete network of four oscillators, they constructed 14 real equilibria.
There are 20  complex equilibria. 
In the 2010s,  Molzahn, {\it et al.}~\cite{MNMH} showed there could be 16 real equilibria
and in 2020, Lindberg {\it et al.}~\cite{LZBL} improved this to 18.
It remains unknown if all 20 complex equilibria can be real.

We have a more complete answer for the enumeration of complex equilibria.
Using the  bihomogeneous B\'ezout bound of an algebraic formulation of the equilibrium equations of \eqref{equ: generalized kuramoto},
Baillieul and Byrnes showed that a network of $n{+}1$ oscillators  has at most $\binom{2n}{n}$ complex equilibria \cite{BaillieulByrnes1982Geometric}.
This upper bound is  attained for generic parameters $\{\omega_i\}$ and $\{ k_{ij} \}$ whose network is a complete graph.

For sparse networks whose underlying graph is not complete, the bihomogeneous B\'ezout bound is not sharp,
as the equations are sparse in the sense of Section~\ref{polyhedral-homotopy}.
Bernstein's Theorem~\ref{thm: Bernstein} is used in \cite{MehtaNguyenTuritsyn2016Numerical} to give a bound that depends upon the underlying
graph.
This is elegantly expressed in terms of the normalized volumes of symmetric edge polytopes.
To a connected graph $G$ we associated its \demph{symmetric edge polytope}, which is defined by
\begin{equation}\label{equ: AP bound}
    \defcolor{\Delta_G}\  \vcentcolon=\  \conv \{ e_i - e_j \mid i\sim j \mbox{ in }G \}\,.
\end{equation}
For a network of $n{+}1$ oscillators this has dimension $n$.
Figure~\ref{F:SEP} shows symmetric edge polytopes for connected graphs on three vertices.
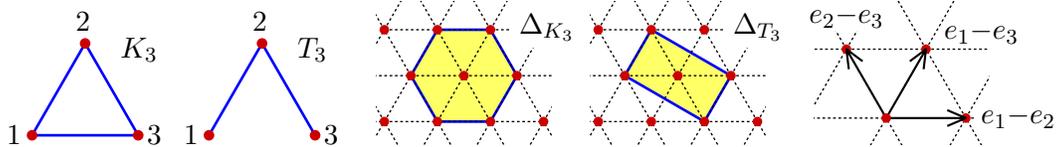
\begin{figure}[htb]
 \centering
  \begin{picture}(63,54)
    \put(6.5,0.5){\includegraphics{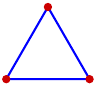}}
    \put(43,33){\small$K_3$}  
    \put(0,0){\small 1}    \put(26,43){\small 2}  \put(53,0){\small 3}
  \end{picture}
  \begin{picture}(63,54)
    \put(6.5,0.5){\includegraphics{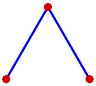}}
    \put(43,33){\small$T_3$}  
    \put(0,0){\small 1}    \put(26,43){\small 2}  \put(53,0){\small 3}
  \end{picture}
  \
  \begin{picture}(77,54)(-2,0)
    \put(0,0){\includegraphics{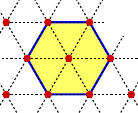}}
    \put(54,42){\small$\Delta_{K_3}$}
  \end{picture}
  \begin{picture}(77,54)(-2,0)
    \put(0,0){\includegraphics{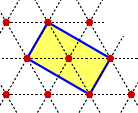}}
    \put(54,42){\small$\Delta_{T_3}$}
  \end{picture}
  \
  \begin{picture}(96,54)(-2,0)
      \put(0,0){\includegraphics{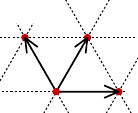}}
      \put(63,8){\small $e_1{-}e_2$}
      \put(49,40){\small $e_1{-}e_3$}
      \put(-2,46){\small $e_2{-}e_3$}
  \end{picture}

  \caption{Connected graphs on three vertices, their symmetric edge polytopes, and the coordinates.}
  \label{F:SEP}
\end{figure}
A result in~\cite{ChenKorchevskaiaLindberg2022Typical} is that for a connected graph $G$ and generic parameters,
there are exactly $n!\vol(\Delta_G)$ complex equilibria.
We may see the bound of six from \cite{BaillieulByrnes1982Geometric} for the network $K_3$ in Figure~\ref{F:SEP}; the hexagon $\Delta_{K_3}$
is composed of six primitive triangles.

This symmetric edge polytope $\Delta_G$ is quite natural and has been independently studied in geometry and number
theory~\cite{DAliDelucchiMichalek2022Many}. 
Table \ref{tab: kuramoto root count} shows examples of the numbers of complex equilibria obtained  through this connection.
\begin{table}[htb]
    \caption{Known results for the generic and maximum complex equilibria of \eqref{equ: generalized kuramoto}}
    \label{tab: kuramoto root count}
    \small
    \centering
    \begin{tabular}{ll}\toprule
        A tree  network with $n+1$ nodes \cite{ChenDavisMehta2018Counting} & $2^{n}$ 
        \\ \midrule
        A cycle network with $n+1$ nodes \cite{ChenDavisMehta2018Counting} & $ (n+1) \binom{n}{ \lfloor n / 2 \rfloor } $ 
        \\ \midrule
        Cycles of lengths
        $2m_1, \ldots, 2m_k$ joined along an edge \cite{DAliDelucchiMichalek2022Many} & 
        $\frac{1}{2^{k-1}} \prod_{i=1}^k m_i \binom{2m_i}{m_i}$ 
        \\ \midrule 
        Cycles of lengths $2m_1{+}1$ and $2m_2{+}1$ joined along an edge \cite{DAliDelucchiMichalek2022Many}  & 
        $(m_1 {+} m_2 {+} 2m_1m_2) \binom{2m_1}{m_1} \binom{2m_2}{m_2}$  
        \\ \midrule
        A wheel graph with $n+1$ nodes for odd $n$ \cite{DAliDelucchiMichalek2022Many}& 
        		$(1-\sqrt{3})^{n} + (1+\sqrt{3})^{n}$
        \\ \midrule
        A wheel graph with $n+1$ nodes for even $n$ \cite{DAliDelucchiMichalek2022Many}& 
        		$(1-\sqrt{3})^{n} + (1+\sqrt{3})^{n} - 2$
        \\
        \bottomrule
    \end{tabular}
\end{table}
Finding exact formulae for other families of networks
remains an active topic.
For trees and cycle networks,
it is possible for all complex equilibria to be real.
It is still unknown if the same holds for other families of networks.

%
%
\subsection{Numerical nonlinear algebra in enumerative geometry}\label{ssec:enum}

Paraphrasing Schubert \cite{Schubert1879}, enumerative geometry is the art of counting geometric figures satisfying conditions imposed by
other, fixed, geometric figures. 
Traditionally, these counting problems are solved by understanding the structure of the space of figures we are to count well enough to
construct their cohomology or Chow rings~\cite{Fulton}, where the computations are carried out.
Numerical nonlinear algebra allows us to actually compute the solutions to a given instance of an enumerative problem and then glean
information about the problem that is not attainable by other means.

While the polyhedral homotopy of Section~\ref{polyhedral-homotopy} based on Bernstein's Theorem may be viewed as a numerical homotopy method
to solve a class of enumerative problems, 
perhaps the first systematic foray in this direction was in~\cite{NSC_HSS} by Sturmfels and coauthors, who exploited structures in the
Grassmannian to give three homotopy methods for solving simple Schubert problems.
These are enumerative problems that ask for the $k$-planes in $\CC^n$ that meet a collection of linear subspaces non-trivially, such as
finding all (462) 3-planes in $\CC^7$ that meet twelve 4-planes~\cite{Sch1886b}.
Their number may be computed using Pieri's formula.
The Pieri homotopy algorithm from~\cite{NSC_HSS} was later used~\cite{Galois_LS} to study Galois groups in Schubert calculus.
This included showing that a particular Schubert problem with 17589 solutions had Galois group the full symmetric group $S_{17589}$.

One of the most famous and historically important enumerative problems is the problem of five conics:
How many plane conics are simultaneously tangent to five given plane conics?
This was posed by Steiner~\cite{St1848} who gave the answer $7776$.
Briefly, a conic $ax^2+bxy+cy^2+dxz+eyz+fz^2=0$ in $\PP^2$ is given by the point $[a,b,c,d,e,f]$ in $\PP^5$, and the condition to be tangent
to a given conic is a sextic in $a,b,\dotsc,f$.
By B\'ezout's Theorem, Steiner expected $6^5=7776$.
The only problem with this approach is that every ``doubled-line conic'', one of the form $(\alpha x+\beta y+\gamma z)^2$, is tangent to
every conic, and thus the B\'ezout count of $7776$ includes a contribution from the doubled-line conics.
Chasles~\cite{Ch1864} essentially introduced the Chow ring of smooth conics to give the correct answer of $3264$~\cite{Kl80}.

Fulton~\cite[p.~55]{Fu96} asked how many of the 3264 conics could be real, later determining that all can be real, but he did not publish
his argument.
His argument  involves deforming an asymmetric pentagonal arrangement of lines and points to prove
the {\it existence} of five real conics having all 3264 tangent conics real.
Ronga, Tognoli, and Vust~\cite{RTV} published a different proof of existence via a delicate local computation near a symmetric arrangement
that had 102 tangent conics, each of multiplicity 32.
Fulton's argument is sketched in~\cite[Ch.~9]{IHP} and Sturmfels with coauthors wrote a delightful article 
``3264 conics in a second''~\cite{3264} in which they give an explicit example of five real conics with 3264 real tangent conics, together
with a proof using certification as in Section~\ref{sec:cert} using numerical nonlinear algebra.
Figure~\ref{F:3264} shows a picture.
\begin{figure}[htb]
 \centering
 \includegraphics[height=220pt]{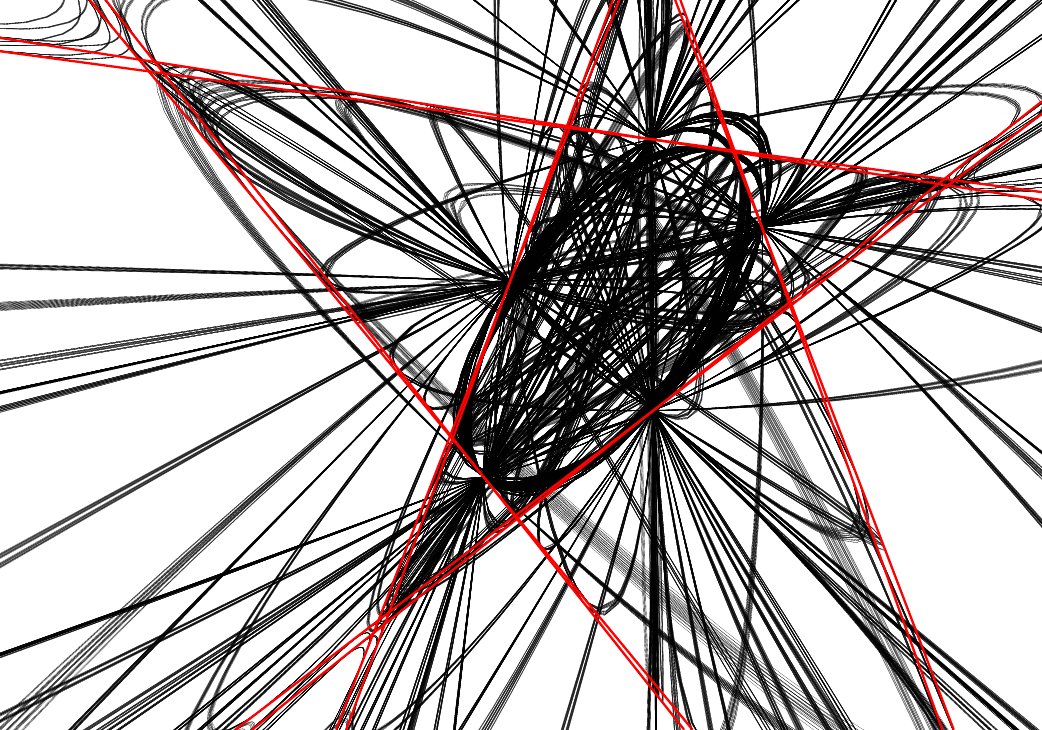}
 \caption{3264 real conics tangent to five (nearly degenerate) conics.}
 \label{F:3264}
\end{figure}

\subsection{Computer vision}\label{ssec:comp_vision}

Computer vision is a field of artificial intelligence that trains computers to interpret and understand the visual world.
Several types of problems in computer vision are amenable to algebraic computational methods.
We shall focus on one type---minimal problems---and one method---homotopy continuation.
Minimal problems are the backbone of the \demph{structure from motion} pipeline that is used for three-dimensional (3D) reconstruction in
applications from medical imaging to autonomous vehicles. 

The chapter ``Snapshot of Algebraic Vision'' in this volume~\cite{AlgVis} treats other types of problems and other algebraic methods,
including symbolic computation.  

All problems amenable to this algebraic analysis share a purely geometric problem at their core.
For computer vision, this often begins with basic projective geometry.
We consider the projective space $\PP^3$ as the 3D world, as it compactifies the Euclidean space $\RR^3$.
A mathematical model of a pin-hole camera \defcolor{$C$} is a projective linear map given by a matrix
\[
   C\ \vcentcolon=\  [R\mid t]\,, \quad R\in\RR^{3\times 3}\ \text{ and }\ t \in \RR^{3\times 1}\,,
\]
which captures the images of world points in the \demph{image plane} $\PP^2$.
A \demph{calibrated} (a.k.a. \demph{pinhole}) camera $C$ has $R \in \SO(3)$.
While this is formulated in the $\PP^3$ compactifying $\RR^3$, for computations we extend scalars to the complex
numbers. 

We may also interpret a calibrated camera as an element of the special Euclidean group $\SE(3)$ acting on $\RR^3$,
the rotation $R$ followed by the translation $t$.
It is convenient to operate in a fixed affine chart on $\PP^3$ and consider the (affine) camera plane as a plane of points with the third
coordinate equal to $1$ in $\RR^3$ (the local affine coordinates of the camera).
The image of a point is obtained by intersecting the image plane with the line going through the point and the center of the camera.
Figure~\ref{fig:5-points} illustrates this as well as the definition of \demph{depth}. 
\begin{figure}[htb]
\centering
  \begin{picture}(250,192)
    \put(0,0){\includegraphics[width=240pt]{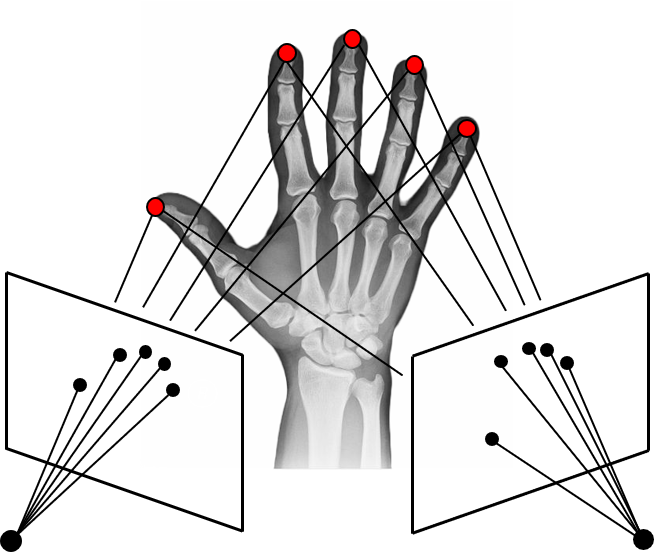}}
    \put(180,152){\small$\begin{pmatrix}\lambda x\\\lambda y\\ \lambda \end{pmatrix}$}
    \put( 64,48){\small$\begin{pmatrix}x\\y\\1\end{pmatrix}$}
  \end{picture}
  \caption{\demph{Five point problem}: given images of five points in two views, find the relative pose $[R\mid t]$ of the two cameras
      with  $t= (t_1,t_2,1)^T$.}
\label{fig:5-points}
\end{figure}
The relative positions of the two cameras is encoded by an element $[R\mid t]$ of the Euclidean special group $\SE(3)$.
The third coordinate of $t$ is set equal to 1 to remove a scaling ambiguity.
In the (3D) coordinate frame of the first camera the image of the tip of the little finger lies in the camera plane and the actual tip is
the point obtained by scaling this by the \demph{depth} $\lambda$.

\subsubsection{Minimal problems}
A reconstruction problem is \demph{minimal}  if it has a finite number of (complex) solutions for general values of the  parameters.
As in Sections~\ref{SS:parameterContinuation} and~\ref{SS:Monodromy}, a minimal problem gives rise to a
branched cover  $\pi\colon\calM\to\calP$, where the base space $\calP$ is called the \demph{problem space} and the total space $\calM$ (incidence
variety) is called the \demph{problem-solution} manifold.
The number of the solutions, which is the degree of the branched cover, is the \demph{degree} of the problem.

\begin{example}\label{ex:3-points}
A classical minimal problem is computing the calibrated camera from three points in space and their image projections~\cite{HZ-2003}. 
One classical formulation~\cite{Grunert-1841} of this problem is as a system of three quadratic polynomial equations
\begin{eqnarray*}
  \|X_1-X_2\|^2\ =\ \|\lambda_1 x_1-\lambda_2 x_2\|^2 \\
  \|X_2-X_3\|^2\ =\ \|\lambda_2 x_2-\lambda_3 x_3\|^2 \\
  \|X_3-X_1\|^2\ =\ \|\lambda_3 x_3-\lambda_1 x_1\|^2
\end{eqnarray*}
in three unknown \demph{depths} $\lambda_1, \lambda_2,\lambda_3$.
The parameters are the three ($i=1,2,3$) world points points $X_i \in \RR^3$ and the points $x_i = (x_{i1},x_{i2},1)^T \in \RR^3$
representing three images in $\PP^2$.  

This formulation implicitly uses that $R$ is orthogonal and preserves the norm.
Recovery of the camera $C=[R\mid t]$ from the depths is an exercise in linear algebra.

This problem has degree eight.
That is, for generic $X_i$ and $x_i$, $i=1,2,3$, the system has eight complex solutions~\cite{DBLP:conf/issac/FaugereMRD08}.
In practice, there are fewer solutions with positive depths $\lambda_i$.
This gives a branched cover of degree eight over the problem manifold $\calP\cong\CC^{15}$.\hfill$\diamond$
\end{example}

We formulate perhaps the most consequential of all 3D reconstruction problems.  

\begin{example}
\label{ex:5-points}
\demph{five point problem} of computing the relative pose of two calibrated cameras from five point correspondences in two images is
featured in Figure~\ref{fig:5-points}.

Consider (paired) images $x_i = (x_{i1},x_{i2},1)^T$, $y_i = (y_{i1},y_{i2},1)^T$ and depths $\lambda_i$ and $\mu_i$, where  $i=1,\dots,5$,
in the first and second cameras, respectively.
Write down all or sufficiently many of $\binom{5}{2}$ same-distance equations 
\[
   \|\lambda_{i} x_{i}-\lambda_{j} x_{j}\|^2\ =\ \|\mu_{i} y_{i}-\mu_{j} y_{j}\|^2, \quad(1 \leq i < j \leq 5)\,,
\]
between the image points, and one same-orientation equation
\begin{multline*}
\qquad\det[
  \lambda_{1} x_{1}-\lambda_{2} x_{2}
  \mid
  \lambda_{1} x_{1}-\lambda_{3} x_{3}
  \mid
  \lambda_{1} x_{1}-\lambda_{4} x_{4}
 ]\\
 =\ \det[
\mu_{1} y_{1}-\mu_{2} y_{2}
\mid
\mu_{1} y_{1}-\mu_{3} y_{3}
\mid
\mu_{1} y_{1}-\mu_{4} y_{4}
]\,.\qquad
\end{multline*}
These determinants are the signed volume of the same tetrahedron (formed by the world points $X_1,\dotsc,X_4$ in different coordinate
frames).
The equality of the volumes is implied by the same-distance equations but not the equality of signs.
Fix one depth, $\lambda_{1} = 1$, to fix the  ambiguity in scaling.
This gives a system of equations in the remaining nine unknown depths $\lambda$ and $\mu$. 

The solution space is the space of vectors of non-fixed depths $\calP=\RR^{9}$.
The projection from the incidence variety/problem-solution manifold to $\calP$ gives a covering of degree $20$.\hfill$\diamond$
\end{example}

As in Section~\ref{SS:Monodromy}, we may analyze the Galois group  of the branched cover.
Decomposing a monodromy group of a minimal problem as shown in \cite{duff2022galois} may lead to an easier 3D reconstruction.
The classical \demph{epipolar geometry} approach to the 
five point problem~\cite[Sect.~9]{HZ-2003} is realized in this way.
This gives a two-stage procedure for the relative pose recovery with the essential stage being a problem of degree $10$.    

\subsubsection{Engineering meets mathematics}
The 
five point problem of Example~\ref{ex:5-points} plays a practical role in solvers for geometric optimization problems in vision based on
RANSAC~\cite{Fischler-Bolles-ACM-1981,Raguram-USAC-PAMI-2013}. 
This problem has many practical solutions based on or inspired by Gr\"obner basis techniques that also use the epipolar geometry
formulation~\cite{Nister-5pt-PAMI-2004}. 

Recently, homotopy continuation has found practical use for minimal problems whose degrees are too high for efficient symbolic computation.
The first step toward practical fast computation was a special solver MINUS~\cite{MINUSwww} based on \texttt{Macaulay2}~\cite{M2} core C++ code for homotopy
continuation and optimized for performance on modern hardware.
Featured in~\cite{TRPLP}, it is deployed on two minimal problems of degrees 216 and 312 involving point as well as line correspondences.
This computes \emph{all} complex solutions and uses postprocessing to filter out irrelevant solutions. 

Unlike~\cite{TRPLP}, the work in~\cite{Hard-problems-CVPR2022} combines a neural network classifier with homotopy continuation.
Rather than compute all solutions and then pick a relevant one, it follows only one continuation path (over $\RR$).
That strategy finds the relevant solution with high probability in several practical scenarios.
It is comparable to state-of-art algorithms for the 
five point problem and exceeds the performance for a 
four point problem.
While matching four points in three calibrated views is not a minimal problem, there is a relaxation of degree 272 that is minimal, and the
solution of the relaxation may be verified by using the original (overdetermined) formulation.


\providecommand{\bysame}{\leavevmode\hbox to3em{\hrulefill}\thinspace}
\providecommand{\MR}{\relax\ifhmode\unskip\space\fi MR }
\providecommand{\MRhref}[2]{%
  \href{http://www.ams.org/mathscinet-getitem?mr=#1}{#2}
}
\providecommand{\href}[2]{#2}

\end{document}